\documentclass [12 pt] {article}
\usepackage{amssymb}
\usepackage{amsthm}

\theoremstyle{definition}

\newtheorem{theorem}{Theorem}[section]

\newtheorem{definition}{Definition}[section]

\newtheorem{example}{Example}[section]

\newtheorem{proposition}{Proposition}[section]

\newtheorem{remark}{Remark}[section]

\newtheorem{lemma}[theorem]{Lemma}

\newcommand{\blabold}[1]{\ensuremath{\mathbb #1}}

\newcommand{\field}[1]{\ensuremath{\mathbb #1}}

\newcommand{\set}[1]{\ensuremath{\mathcal #1}}

\newcommand{\vect}[1]{\ensuremath{\mathbf #1}}

			\newcommand{\vectorlist}[3]{\ensuremath{\mathbf #1_{#2},  \ldots, \mathbf #1_{#3}}} 

			\newcommand{\scalarlist}[2]{\ensuremath{#1_1, \ldots, #1_{#2}}}

\begin{document}


\title{Schur's Lemma for Coupled Reducibility and Coupled Normality}

\author{
{\footnote{ {\bf Funding}: The work of D. Lahat and C. Jutten was supported by the project CHESS, 2012-
ERC-AdG-320684. GIPSA-lab is a partner of the LabEx PERSYVAL-Lab (ANRÐ11-LABX-0025).}}
Dana Lahat\thanks{  IRIT, Universit\'{e} de Toulouse, CNRS, Toulouse, France (Dana.Lahat@irit.fr).          
This work was carried out when D. Lahat was with GIPSA-lab, Grenoble, France.}      
 \and Christian Jutten\thanks{Univ. Grenoble Alpes, CNRS, Grenoble INP, GIPSA-lab, 38000 Grenoble, France. C. Jutten is
a senior member of Institut Univ. de France (christian.jutten@gipsa-lab.grenoble-inp.fr).} 
\and Helene Shapiro\thanks{Department of Mathematics and Statistics, Swarthmore College, Swarthmore, PA 19081 (hshapir1@swarthmore.edu).}}

\maketitle

\begin{abstract}
Let $\set A = \{A_{ij} \}_{i, j \in \set I}$, where \set I is an index set,   be a doubly indexed family of matrices, where 
$A_{ij}$ is $n_i \times n_j$.  For each $i \in \set I$, let $\set V_i$ be an $n_i$-dimensional vector space.  
We say \set A is {\em reducible in the coupled sense} if there exist subspaces, $\set U_i \subseteq \set V_i$, with 
$\set U_i \neq \{0\}$ for at least one $i \in \set I$, and  $\set U_i \neq \set V_i$ for at least one $i$, such that 
$A_{ij} (\set U_j) \subseteq \set U_i$ for all~$i, j$.
Let $\set B = \{B_{ij} \}_{i, j \in \set I}$ also be a doubly indexed family of matrices, where $B_{ij}$ is $m_i \times m_j$.
For each $i \in \set I$, let $X_i$ be a matrix of size $n_i \times m_i$.  Suppose 
$A_{ij} X_j = X_i B_{ij}$ for all~$i, j$.  We prove versions of Schur's Lemma for  \set A, \set B satisfying 
coupled irreducibility conditions. 
We also consider a refinement of Schur's Lemma for sets of normal matrices and prove corresponding versions for \set A, \set B satisfying  coupled normality 
and coupled irreducibility conditions.  
\end{abstract}

\section{Introduction}

Let $K$ be a positive integer and let \scalarlist {n}{K} and \scalarlist {m}{K} be positive integers.
Consider two doubly indexed families of matrices,  $\set A = \{A_{ij} \}_{i, j = 1}^K$ and $\set B = \{B_{ij} \}_{i, j = 1}^K$, where $A_{ij}$ is $n_i \times n_j$
and  $B_{ij}$ is $m_i \times m_j$.  Put $N = \sum_{i=1}^K n_i$ and $M = \sum_{i=1}^K m_i$.  
 Arrange the $A_{ij}$'s into an $N \times N$ matrix, $A$, with $A_{ij}$ in block $i,j$ of $A$. 
\[ A = \pmatrix{A_{11} & A_{12} & \cdots & A_{1K} \cr
A_{21} & A_{22} & \cdots & A_{2K} \cr
\vdots & \vdots & & \vdots \cr
A_{K1} & A_{K2} & \cdots & A_{KK} \cr }. \]
Similarly, form an $M \times M$ matrix $B$ with $B_{ij}$ in block $i, j$.  Let $X_i$  be an $n_i \times m_i$ matrix and form an 
$N \times M$ matrix X with blocks \scalarlist{X}{K} down the diagonal and zero blocks elsewhere.  Thus, 
\[ X = \pmatrix{X_{1} & 0 & 0  & \cdots & 0 \cr
0  & X_{2} & 0 & \cdots & 0 \cr
0 & 0 & X_3  & \cdots & 0 \cr 
\vdots & \vdots  & \vdots  &  & \vdots \cr
0 & 0 & 0 & \cdots & X_{K} \cr }, \]
where the  $0$ in position $i,j$ represents an  $n_i \times m_i$ block of zeroes.   Then $AX = XB$ if and only if 
  \begin{equation}
  \label{startingbasiccoupling}
  A_{ij} X_j = X_i B_{ij},
  \end{equation}
for all $i, j = 1, \ldots, K$. 
We may rewrite $AX = XB$ as $AX - XB = 0$, a homogeneous Sylvester equation~\cite{Syl84}.  

We define several versions of coupled reducibility and prove corresponding versions of Schur's Lemma~\cite{Schur} for pairs~\set A,~\set B.
Imposing coupled irreducibility constraints on \set A and \set B restricts the possible solutions, \scalarlist {X}{K}, to the equations~(\ref{startingbasiccoupling}). We also discuss a refinement  of Schur's Lemma for normal matrices, and 
prove corresponding versions for \set A, \set B satisfying coupled normality conditions. 

The system of coupled matrix equations~(\ref{startingbasiccoupling}) arises in a recent model for multiset data analysis~\cite{LJ15, LJ16}, called 
{\em joint independent subspace analysis}, or JISA. This model consists of $K$ datasets, where each dataset is an unknown mixture of several latent stochastic multivariate signals. The blocks of $A$ and~$B$ represent statistical links among these datasets. More specifically,~$A_{ij}$ and $B_{ij}$ represent statistical correlations among latent signals in the $i$th and $j$th datasets. The multiset joint analysis framework, as opposed to the analysis of K distinct unrelated datasets, arises when a sufficient number of cross-correlations among datasets, i.e., $A_{ij}$ and
 $B_{ij}, i \neq j$, are not zero. It turns out~\cite{LaJu15,  LJ18,  LaJu18} that the identifiability and uniqueness of this model, in its simplest form, boil down to characterizing the set of solutions to the system of matrix equations~(\ref{startingbasiccoupling}), when the cross-correlations among the latent signals are subject to coupled (ir)reducibility. In other words, the coupled reducibility conditions that we introduce in this paper can be attributed with a physical meaning, and can be applied to real-world problems. We refer the reader 
 to~\cite{LaJu15, LaJu18} for further details on the JISA model, and on the derivation of~(\ref{startingbasiccoupling}). In this paper, we consider scenarios more general than those required to address the signal processing problem in~\cite{LaJu15,  LaJu18}. The analysis in Section 6, which focuses on coupled normal matrices, is inspired by the JISA model 
 in~\cite{LaJu15, LaJu18}, in which the matrices $A$ and $B$ are Hermitian, and thus a special case of coupled normal.
A limited version of some of the results in this manuscript was presented orally in, e.g.,~\cite{LahJut15, Lahat16, ALJ16} and in a technical report~\cite{LaJu15}; however, they were never published.

While motivated by the matrix equation, $AX = XB$,  the main definitions, theorems, and proofs do not require 
$K$ to be finite.  Thus, for a general index set, \set I, we consider doubly indexed families, $\set A = \{A_{ij} \}_{i, j \in \set I}$  and 
$\set B = \{B_{ij} \}_{i, j \in \set I}$.  
For the situation described above, $\set I = \{1, 2, \ldots, K\}$.

We use  \field F to denote the field of scalars.  For some results, \field F can be any field.  
For results involving unitary and normal matrices, 
$\field F = \blabold C$, the field of complex numbers. 
For each $i \in \set I$, let $n_i$ and $m_i$ be positive integers, let $\set V_i$ be an
$n_i$-dimensional vector space over \field F, and let $\set W_i$ be an $m_i$-dimensional vector space over \field F.  
(Essentially, $\set V_i = \field F^{n_i}$ and $\set W_i = \field F^{m_i}$.)   For all $i, j \in \set I$, let $A_{ij}$ be an $n_i \times n_j$ matrix 
and let $B_{ij}$ be $m_i \times m_j$.  View $A_{ij}$ as a linear transformation from $\set V_j$ to $\set V_i$, and 
$B_{ij}$ as a linear transformation from $\set W_j$ to $\set W_i$.  For each $i \in \set I$, let $X_i$ be an $n_i \times m_i$ matrix; 
view $X_i$ as a linear transformation from $\set W_i$ to $\set V_i$.  We are interested in families \set A, \set B satisfying 
the equations~(\ref{startingbasiccoupling}) for all $i, j \in \set I$.  

For some results, all of the $n_i$'s will be equal, and all of the $m_i$'s will be equal.  In this case, we use $n$ for the 
common value of the $n_i$'s and $m$ for the common value of the $m_i$'s.   We then set $\set V = \field F^n$ and 
$\set W = \field F^m$.   Each $X_i$ is then $n \times m$.  For $\set I = \{1, \ldots, K\}$, we have $N = Kn$, and $M = Km$.  
All of the blocks $A_{ij}$ in $A$ are $n \times n$, while all of the blocks $B_{ij}$ in $B$ are $m \times m$.

 Section~\ref{sectionreduSchur} reviews the usual matrix version of Schur's Lemma and its proof.  Section~\ref{coupledreduc} defines coupled reducibility and 
 two restricted versions, called {\it proper} and {\it strong} reducibility. 
 Section~\ref{coupledSchurLemma} states and proves versions of Schur's Lemma for coupled reducibility and proper reducibility, Theorem~\ref{maintheorem}.  
 In section~\ref{sectionstrongreduc}, we define some graphs associated with the pair \set A, \set B and use them for versions of
 Schur's Lemma corresponding to strongly coupled reducibility, Theorems~\ref{Schurlemmaversionthree} and~\ref{Schurlemmaversionfour}.  
 Section~\ref{normalcoupled} deals  with a refinement of Schur's Lemma for sets of normal matrices and 
 corresponding versions for pairs \set A, \set B which are coupled normal, Theorem~\ref{couplednormalirred}.  The Appendix presents examples to support some claims made
 in Section~\ref{coupledreduc}.


\section{Reducibility and Schur's Lemma}
\label{sectionreduSchur}
We begin by reviewing the ordinary notion of reducibility for a set of linear transformations and Schur's Lemma.

\begin{definition} 
\label{transformationreducible}
A set, \set T,  of linear transformations,  on an $n$-dimensional vector space, \set V, is {\it reducible} if there is a proper, non-zero 
subspace \set U of~\set V such that $T(\set U) \subseteq \set U$ for all $T \in \set T$.  If \set T is not reducible, we say it is {\it irreducible}.  
\end{definition}

The subspace \set U is an invariant subspace for each transformation $T$ in~\set T.     We say \set T is {\it fully reducible} if it is possible to decompose \set V as a direct sum
 $\set V = \set U \oplus \hat{\set U}$, where \set U and $\hat{\set U}$ are both nonzero, proper invariant subspaces of~\set T.

Alternatively, one can state this in matrix terms.  The linear transformations in \set T may be represented as $n \times n$ matrices, relative to a choice of basis for 
\set V.   Let~$d$ be the dimension of \set U; choose a basis for \set V in which the first~$d$ basis vectors are a basis for~\set U.  Since  \set U is an invariant subspace of each~$T$ in~\set T, the matrices representing \set T relative to this basis are block upper triangular with square diagonal blocks of sizes $d \times d$ and 
$(n-d) \times (n-d)$.  The first diagonal block, of size $d \times d$, represents the action of the transformation on the invariant subspace~\set U.  
The $(n-d) \times d$ block in the lower left hand corner consists of zeroes.  Since changing basis is equivalent to applying a matrix similarity, we have the following matrix version of Definition~\ref{transformationreducible}.

\begin{definition} (Matrix version of reducibility.)
\label{matrixreducible}
A set  \set M  of $n \times n$ matrices is {\it reducible} if, for some $d$, with $0 < d < n$, there is a nonsingular matrix~$S$ such that,
for each $A$ in \set M, the matrix $S^{-1} AS$ is block upper triangular with  square diagonal blocks of sizes $d \times d$ and $(n-d) \times (n-d)$. 
\end{definition}

When \set T is fully reducible, we can use a basis for \set V in which the first $d$ basis vectors are a basis for $\set U$, and the remaining $n-d$ basis vectors are 
a basis for~$\hat{\set U}$.  The corresponding matrices for \set T are then block diagonal, with diagonal blocks of sizes $d \times d$ and $(n-d) \times (n-d)$.  

If $\field F = \blabold C$ and \set M is reducible, the $S$ in Definition~\ref{matrixreducible} can be chosen to be a unitary matrix, by using an orthonormal basis 
for $\blabold C^n$ in which the first~$d$ basis vectors are an orthonormal basis for \set U.  If \set M is fully reducible, and the subspaces $\set U$ and $\hat{\set U}$ are 
orthogonal, use an orthonormal basis of $\set U$ for the first $d$ columns of $S$ and an orthonormal basis of $\hat{\set U}$ for the remaining $n -d$ columns. 
Then $S$ is unitary, and $S^{-1}M S$ is block diagonal, with diagonal blocks of sizes $d \times d$ and $(n-d) \times (n-d)$.

The following fact is well known; we include the proof because the idea is used later.  For this fact, we assume we are working over \blabold C, or at least over a field that contains the eigenvalues of the transformations.  

\begin{proposition}
 Let  \set M be an irreducible set of $n \times n$ complex matrices.  
 Suppose the $n \times n$ matrix $C$ commutes with every element of~\set M.  Then $C$ is a scalar matrix. 
 \end{proposition}
 
 \begin{proof}
 Let $\lambda$ be an eigenvalue of $C$, and let  $\set U_{\lambda}$ denote the corresponding eigenspace. 
 Let $A \in \set M$.  For any $\vect v \in \set U_{\lambda}$, we have $C(A \vect v) = A(C \vect v) = \lambda (A\vect v)$. 
 Hence  $A \vect v$ is in $\set U_{\lambda}$, and so $\set U_{\lambda}$ is invariant under each element of \set M.  Since 
 an eigenspace is nonzero, and \set M is irreducible, we must have $\set U_{\lambda} = \blabold C^n$.  Hence~$C = \lambda I_n$. 
 \end{proof}

 Schur's Lemma plays a key role in group representation theory.  It is used to establish uniqueness of 
 the decomposition of a representation of a finite group into a sum of irreducible representations.    
 However, one need not have a matrix group; the result holds for irreducible sets of matrices.
 We include the usual proof~\cite{Art11}, because the same idea is used to prove our versions for coupled reducibility.

\begin{theorem} [Schur's Lemma]
\label{SchurLemma} 
 Let $\{ A_i\}_{i \in \set I}$ be an irreducible family of $n \times n$ matrices, and let
$\{ B_i \}_{i \in \set I}$ be an irreducible family of $m \times m$ matrices.  Suppose $P$ is an $n \times m$ matrix such that 
$A_i P = P B_i$ for all $i \in \set I$.  Then, either $P = 0$, or $P$ is nonsingular; in the latter case we must have $m = n$.
 For matrices of complex numbers, if  $A_i = B_i$ for all $i \in \set I$, then $P$ is a scalar matrix.  
\end{theorem}

\begin{proof}
View the $A_i$'s as linear transformations on an $n$-dimensional vector space \set V, and 
the $B_i$'s as linear transformations on an $m$-dimensional vector space \set W.  
The $n \times m$ matrix  $P$ represents a linear transformation from~\set W to~\set V.  So $ker(P)$ is a subspace  of \set W and 
$range(P)$ is a subspace of \set V.  

Let $\vect w \in ker(P)$.  Then $P (B_i \vect w) = A_i P \vect w = 0$.  Hence, $ker(P)$ is invariant under $\{ B_i\}_{i \in \set I}$.
Since $\{ B_i \}_{i \in \set I}$ is irreducible,  $ker(P)$ is either  $\{0\}$ or  \set W.  In the latter case, $P = 0$.  If $P \neq 0$, then 
$ker(P) = \{0\}$.  
Now consider the range space of $P$.  For any $\vect w \in \set W$, we have $A_i( P \vect w)=P(B_i \vect w)~\in range(P)$,  
so the range space of~$P$ is invariant under $A_i$ for each~$i$.  Since $\{ A_i\}_{i \in \set I}$ is irreducible,~$range(P)$ is either 
$\{0\}$ or \set V.  But we are assuming $P \neq 0$ so~$range(P) = \set V$.  Since we also have $ker(P) = \{0\}$, the matrix $P$ must be 
nonsingular and~$m = n$.   

If $A_i = B_i$ for 
all $i \in \set I$, then $P$ commutes with each $A_i$.  If each $A_i$ is  a complex matrix, then, since $\{ A_i\}_{i \in \set I}$ is irreducible, $P$ must be a scalar matrix.  
\end{proof}

\noindent
For nonsingular $P$, we have $P^{-1} A_i P = B_i$ for all $i \in \set I$, so  $\{ A_i \}_{i \in \set I}$ and~$\{ B_i \}_{i \in \set I}$ are simultaneously similar.


\section{Coupled Reducibility}
\label{coupledreduc}

For simultaneous similarity of  $\{ A_i \}_{i \in \set I}$ and $\{ B_i \}_{i \in \set I}$, there is a
nonsingular matrix,~$P$, such that $P^{-1} A_i P = B_i$ for all~$i$.  
We now define a ``coupled" version of similarity for two doubly indexed families,  \set A and  \set B, with $n_i = m_i$ for all $i \in \set I$. 
In this case, $A_{ij}$ and $B_{ij}$ are matrices of the same size, and in the equations~(\ref{startingbasiccoupling}), each $X_i$ is a square matrix.    
 
 \begin{definition}
 Let  $\set A = \{A_{ij} \}_{i, j \in I}$ and $\set B = \{B_{ij} \}_{i, j \in I}$, where $n_i = m_i$ for all $i \in \set I$.  We say \set A and \set B are 
 {\it similar in the coupled sense} if there exist nonsingular matrices $\{ T_i \}_{i \in I}$, where $T_i$ is $n_i \times n_i$, such that 
$T_i^{-1} A_{ij} T_j = B_{ij}$
 for all $i, j \in \set I$.  
 \end{definition}
 For a finite index set $\set I = \{1, \ldots, K\}$, this can be stated in terms of the matrices, $A$ and  $B$.  
 Let $T$ be the block diagonal matrix $ T_1 \oplus T_2 \oplus \cdots \oplus T_K$.  
  Then $AT = TB$ if and only if $A_{ij}T_j = T_i B_{ij}$ for all $i, j$.  Hence,~\set A and \set B are similar in the coupled sense 
  if and only if $T$ is nonsingular and~$T^{-1} A T = B$.   
 
 We define several versions of ``reducible in the coupled sense"  for a doubly indexed family \set A.  The basic idea is that there are subspaces, 
 $\{ \set U_i \}_{i \in \set I}$, where $\set U_i \subseteq \set V_i$, such that $A_{ij}(\set U_j) \subseteq \set U_i$ for all $i, j \in \set I$.  
 This holds trivially when $\set U_i$ is zero for all $i$, and when
 $\set U_i = \set V_i$  for all $i$, so we shall insist that at least one subspace is nonzero, and at least one is not $\set V_i$. 
 We are also interested in two more restrictive versions:  the case where at least one $\set U_i$ is a nonzero, proper subspace, and the 
 case where every $\set U_i$ is a nonzero, proper subspace of $\set V_i$. 
    
 \begin{definition}
 \label{generalcoupledreduc}
 Let  $\set A = \{A_{ij} \}_{i, j \in I}$ where $A_{ij}$ is $n_i \times n_j$.  We say  \set A is {\em reducible in the coupled sense} if there 
 exist subspaces $\{ \set U_i \}_{i \in \set I}$, where $\set U_i \subseteq \set V_i$,
 such that the following hold.   
 \begin{enumerate}
 \item For at least one $i$, we have  $\set U_i \neq \{0\}$.
 
 \item For at least one $i$, we have $\set U_i \neq \set V_i$.  
 
 \item For all $i, j \in \set I$ we have
 $ A_{ij} (\set U_j) \subseteq \set U_i.$
 \end{enumerate}
 We say $\{ \set U_i \}_{i \in \set I}$ is a reducing set of subspaces for \set A, or that \set A is reduced by $\{ \set U_i \}_{i \in \set I}$.
  If \set A is not reducible in the coupled sense, we say it is {\em irreducible} in the coupled sense.  
 We say \set A is {\em properly reducible in the coupled sense} if at least one $\set U_i$ is a nonzero, proper subspace of $\set V_i$.  
 We say \set A is {\em strongly reducible in the coupled sense} if every $\set U_i$   is a nonzero, proper subspace of $\set V_i$. 
\end{definition}

\begin{remark} If $n_i = 1$ for all $i$, the one-dimensional spaces $\set V_i$ have no nonzero proper invariant subspaces, so \set A cannot be 
properly or strongly irreducible.  If $K = 1$, then \set A consists of a single $n \times n$ matrix.  
\end{remark} 

\begin{remark} If $n_i = n$ for all $i$, and the subspaces $\{ \set U_i \}_{i \in \set I}$ are all the same nonzero proper subspace, i.e, for all $i$, we have 
$\set U_i = \set U$ where \set U is a nonzero, proper subspace of \set V, then $\set A $ is reducible in the 
ordinary sense, given in Definition~\ref{transformationreducible}.  
\end{remark} 
Note the following facts. 
\begin{enumerate}
\item If $\set U_j = \{0\}$,  then $ A_{ij} (\set U_j) \subseteq \set U_i$ holds for any $A_{ij}$ and any $\set U_i$.  
\item If $\set U_i = \set V_i$, then $ A_{ij} (\set U_j) \subseteq \set U_i$ holds for any $A_{ij}$ and any $\set U_j$. 
\item If $\set U_j = \set V_j$ and $\set U_i = \{0\}$, then $A_{ij} = 0$.
\item For $i = j$, we have $A_{ii}(\set U_i) \subseteq \set U_i$, so $\set U_i$ is an invariant subspace of $A_{ii}$. 
\end{enumerate}

An equivalent matrix version of Definition~\ref{generalcoupledreduc} is obtained by choosing an appropriate basis for each $\set V_j$.    
 Let $d_j$ be the dimension of the subspace $\set U_j$.  We have $ 0 \leq d_j \leq  n_j$.  If $d_j$ is positive, let 
  $\vect v_{j,1}, \ldots, \vect v_{j, d_j}$ be a basis for $\set U_j$ and let $T_j$ be a nonsingular $n_j \times n_j$ matrix which has 
  $\vect v_{j,1}, \ldots, \vect v_{j, d_j}$  in the first $d_j$ columns.  If $d_j = 0$, we may use any nonsingular $n_j \times n_j$ matrix for $T_j$.  
  Set $B_{ij} = T_i^{-1} A_{ij} T_j$; equivalently, 
  \[
 A_{ij} T_j = T_i B_{ij}.
 \]
  The first $d_j$ columns of $T_j$ are a basis for $\set U_j$, so 
 $A_{ij} (\set U_j) \subseteq \set U_i$ tells us the first~$d_j$ columns of 
  $A_{ij} T_j$ are in~$\set U_i$.   Hence, the first $d_j$ columns of $T_i B_{ij}$ are in~$\set U_i$, 
  so by the definition of $T_i$, each of the first $d_j$ columns of $T_i B_{ij}$ is a linear combination of the first 
  $d_i$ columns of $T_i$.  Therefore, each of the first~$d_j$ columns of $B_{ij}$ will 
  have zeroes in all entries below row~$d_i$, and the lower left hand corner of $B_{ij}$ is a block of zeroes of size $(n_i - d_i) \times d_j$.  When 
  $0 < d_i < n_i$ and $0 <  d_j < n_j$,  the matrix 
  $B_{ij}$ has the form
  \begin{equation}
  \label{basicblockform}
 B_{ij} =  \pmatrix{C \ \ \ \ \ \ & D \cr  0_{(n-d_i) \times d_j} & E},
  \end{equation}
  where $C$ is size $d_i \times d_j$ and represents the action of $A_{ij}$ on the subspace $\set U_j$.  The zero block in the lower left hand block has 
  size $(n- d_i) \times d_j$, while $D$ is $d_i \times (n_j-d_j)$ and $E$ is $(n_i-d_i) \times (n_j-d_j)$.  
  
  \begin{remark}  
  The block matrix~(\ref{basicblockform}) has a block of zeroes in the lower left hand corner of size $(n-d_i ) \times d_j$.  We also use this terminology
  for the cases $d_i = 0, \ d_i = n_i, \ d_j = 0, \ d_j =n_j$, in which case we mean the following.  
  If $d_i = 0$, the first $d_j$ columns of $B_{ij}$ are zero.  If $d_i = n_i$  there is no restriction on the form of $B_{ij}$.  
   If $d_j = 0$ there is no restriction on the form of $B_{ij}$.   If $d_j = n_j$ the last $n - d_i$ rows of $B_{ij}$ are zero.  
\end{remark}

  Conversely, if each $B_{ij} = T_i^{-1} A_{ij} T_j$ has the block form in~(\ref{basicblockform}), define  $\set U_i$ to be the subspace spanned by the 
  first $d_i$ columns of $T_i$.  The subspaces $\{ \set U_i \}_{i \in \set I}$ then satisfy $A_{ij}(\set U_j) \subseteq \set U_i$.    Hence, we have the following equivalent matrix form of Definition~(\ref{generalcoupledreduc}).
  
  \begin{definition} [Matrix version of coupled reducibility]
 \label{coupledreducmatrixform}
 Let  $\set A = \{A_{ij} \}_{i, j \in \set I}$. We say \set A is {\it reducible in the coupled sense},  or, {\it reducible by coupled similarity}, if there exist integers 
$\{d_i\}_{i \in \set I}$, with $0 \leq d_i  \leq n_i$, and nonsingular $n_i \times n_i$ matrices $T_i$, such that the following hold. 
\begin{enumerate}
\item At least one $d_i$ is positive. 
\item At least one $d_i$ is less than $n_i$. 
\item Each matrix $B_{ij} = T_i^{-1} A_{ij} T_j$ has a block of zeroes in the lower
 left hand corner of size $(n_i-d_i ) \times d_j$. 
 \end{enumerate}
 We say \set A is {\it properly reducible in the coupled sense} if $0 < d_i < n_i$ for at least one value of $i$. 
 We say \set A is {\it strongly reducible in the coupled sense} if $0 < d_i < n_i$ for every~$i $. 
  \end{definition}

  Full reducibility by coupled similarity occurs when, for each $i$, there is also a subspace $\hat {\set U_i}$ such that 
  $\set V_i  = \set U_i \oplus \hat {\set U_i}$,  and $A_{ij} (\hat {\set U_j})  \subseteq  \hat  {\set U_i}$
 for all $i, j \in \set I$.   
 For the corresponding matrix version, use a basis for $\set U_j$ in the first $d_j$ columns of $T_j$ and a basis for $\hat {\set U_j}$ in the 
 remaining $n_j - d_j$ columns.  For $0 < d_i < n_i$ and $0 < d_j < n_j$, the matrix $B_{ij} = T_i^{-1} A_{ij} T_j $ has the block form
  \begin{equation}
  \label{fullyreducibleblock}
 B_{ij} =  \pmatrix{C \ \ \ \ \ \  & 0_{d_i \times (n-d_j)} \cr 0_{(n-d_i) \times d_j} & E}.
  \end{equation}
The $d_i \times d_j$ matrix $C$ represents the action of $A_{ij}$ on $\set U_j$ and the $(n_i-d_i) \times (n_j-d_j)$ matrix $E$ represents the action of 
  $A_{ij}$ on $\hat {\set U_j}$.

  For the field of complex numbers we have unitary versions.  
 
 \begin{definition} [Unitary version of reducible in the coupled sense]
 \label{coupledreducmatrixformunitary}
 Let  \set A be a family of complex matrices. 
 We say  \set A is {\it unitarily reducible in the coupled sense} if the conditions of 
 Definition~\ref{coupledreducmatrixform} are satisfied with unitary matrices~$T_i$.   
  \end{definition}

  For complex \set A, reducibility by coupled similarity implies reducibility by coupled unitary similarity. 
  Simply use an orthonormal basis for each $\set U_i$, and extend it to an orthonormal basis for $\set V_i$ to obtain a unitary matrix for $T_i$.  
  If~\set A is fully reducible, and  $\set U_i$ and $\hat {\set U_i}$ are orthogonal subspaces,
 then, for each $\set V_i = \set U_i \oplus \hat {\set U_i}$, we can form a unitary matrix $T_i$   using 
an orthonormal basis for~$\set U_i$ for the first~$d_i$ columns and an orthonormal basis for~$\hat {\set U_i}$ for the remaining $n_i - d_i$ columns.  
Each $B_{ij}$ 
 then has the block form~(\ref{fullyreducibleblock}).
 
 Unitary reducibility matters in  the JISA model, because $A$ and $B$ are correlation matrices, and the appropriate linear change of variable 
 leads to a congruence, rather than a similarity.  
When $T_i$ is unitary, $T_i^{-1} = T_i^*$. For $T = T_1 \oplus T_2 \oplus \cdots  \oplus T_k$ we then have 
 $T^{-1} A T = T^* A T$.   
 
From the definition, it is clear that if \set A is strongly reducible, then it is also properly reducible, and if it is properly reducible, it is 
reducible.   We introduce some notation.   Fix an index set, \set I.  Use $\vert \set I \vert$ to denote the size of~\set I; when \set I is 
a finite set with $K$ elements, we assume $\set I = \{1, 2, \ldots, K\}$.   Fix a family $\{n_i\}_{i \in \set I}$ of positive integers, and a field \field F. 
Consider the set of all $\set A = \{A_{ij}\}_{i, j \in \set I}$, where $A_{ij}$ is an $n_i \times n_j$ matrix with entries from \field F.  We use 
$Red(\field F,    \{n_i\}_{i \in \set I})$ to denote the set of all such families \set A which are reducible in the coupled sense.   
  We use $PropRed(\field F, \{n_i\}_{i \in \set I})$ for the set of all such \set A which are properly reducible in the coupled sense, and  
  $StrRed(\field F, \{n_i\}_{i \in \set I})$ for the set of all such \set A that are strongly reducible in the coupled sense.  
  When all $n_i$'s have the same value, $n$, and $\vert \set I \vert  = K$, we use the notations $Red(\field F, n, K), \  PropRed(\field F, n, K)$, 
  and $StrRed(\field F, n, K)$.

 When $n_i = 1$, the space $\set V_i = \field F$ is one dimensional and has no nonzero proper subspaces. 
  Hence, $StrRed(\field F, \{n_i\}_{i \in \set I})$ is the empty set if $n_i = 1$ for some $i$, and $PropRed(\field F, \{n_i\}_{i \in \set I})$
  is the empty set when $n_i = 1$ for all $i$.

  From Definition~\ref{generalcoupledreduc} it is obvious that
 \begin{equation}
 \label{subsetinclusions}
 StrRed(\field F, \{n_i\}_{i \in \set I}) \subseteq PropRed(\field F, \{n_i\}_{i \in \set I}) \subseteq Red(\field F, \{n_i\}_{i \in \set I}).
 \end{equation}  
 Using the superscript ``$C$" to indicate the complement of a set, we then have 
 \begin{equation}
 \label{subsetcomplements}
 Red^C(\field F, \{n_i\}_{i \in \set I}) \subseteq PropRed^C(\field F, \{n_i\}_{i \in \set I}) \subseteq StrRed^C(\field F, \{n_i\}_{i \in \set I}).
 \end{equation}  
 The symbol ``$\subseteq$" means ``subset of or equal to."  We use~``$\subset$" to indicate ``proper subset of."  
 One might expect that ``$\subseteq$" can generally be replaced by~``$\subset$" in~(\ref{subsetinclusions}) and~(\ref{subsetcomplements}).
 This is correct when \set I has at least four elements, and $n_i \geq 2$ for at least one value of $i$.   
 Furthermore, for $\vert \set I \vert \geq 2$, we have $StrRed(\field F, \{n_i\}_{i \in \set I}) \subset PropRed(\field F, \{n_i\}_{i \in \set I})$, provided 
 $n_i \geq 2$ for at least one~$i$.  
 However, for $\vert \set I \vert = 2$ and $\vert \set I \vert =  3$, whether $PropRed(\field F, \{n_i\}_{i \in \set I})$ is equal to, or is a proper subset of,   
 $Red(\field F, \{n_i\}_{i \in \set I})$ depends on the 
 field \field F, and on the~$n_i$'s.  
  The appendix treats this in more detail.  Here is a summary of what is shown there. 
 \begin{enumerate}
 \item  For any field \field F, if $\vert \set I \vert  \geq 4$ and $n_i \geq 2$ for at least one $i$,  
 \[StrRed(\field F, \{n_i\}_{i \in \set I}) \subset PropRed(\field F, \{n_i\}_{i \in \set I}) \subset Red(\field F, \{n_i\}_{i \in \set I}).\] 
 Consequently,  
 \[Red^C(\field F, \{n_i\}_{i \in \set I}) \subset PropRed^C(\field F, \{n_i\}_{i \in \set I}) \subset StrRed^C(\field F, \{n_i\}_{i \in \set I}).\]
 \item For any field \field F, if $\vert \set I \vert  \geq 2$ and $n_i \geq 2$ for at least one $i$,
 \[StrRed(\field F, \{n_i\}_{i \in \set I}) \subset PropRed(\field F, \{n_i\}_{i \in \set I}).\]
  \item  If  \field F is algebraically closed and $n \geq 2$, then 
  \[ PropRed(\field F, n, 2) = Red(\field F, n, 2) \  {\rm and} \  
  PropRed(\field F, n, 3) = Red(\field F, n, 3).\]
  \item For the field, \blabold R, of real numbers, when $n = 2$, we have 
   \[ PropRed(\blabold R, 2, 2) \subset Red(\blabold R, 2, 2) \ {\rm and} \   
   PropRed(\blabold R, 2, 3) \subset Red(\blabold R, 2, 3).\]
  For $n \geq 3$, 
   \[ PropRed(\blabold R, n, 2) = Red(\blabold R, n, 2) \ {\rm and} \  
   PropRed(\blabold R, n, 3) = Red(\blabold R, n, 3).\]
    \end{enumerate}
  %


 \section {A coupled version of Schur's Lemma}
 \label{coupledSchurLemma}
 
 The main result of this section is Theorem~\ref{maintheorem}, a coupled version of Schur's Lemma for reducibility and proper reducibility.  
 Section~\ref{sectionstrongreduc} deals with the more complicated version for strong reducibility.
  
 Consider families $\set A = \{A_{ij} \}_{i, j \in \set I}$ and $\set B = \{B_{ij} \}_{i, j \in I}$, where $A_{ij}$ is $n_i \times n_j$ and 
 $B_{ij}$ is $m_i \times m_j$, linked by equations $A_{ij} X_j = X_i B_{ij}$, where 
  $X_i$ is $n_i \times m_i$.   Recall that $A_{ij}$ is a linear transformation from $\set V_j$ to $\set V_i$, and
  $B_{ij}$ is a linear transformation from $\set W_j$ to $\set W_i$   
   The matrix $X_i$ is a linear transformation from $\set W_i$  to $\set V_i$.  Note that $ker(X_i)$ is a subspace of $\set W_i$ and  $range(X_i)$ is 
   a subspace of $\set V_i$.  

Reviewing the proof of Schur's Lemma (Theorem~\ref{SchurLemma}), the key facts are that $ker(P)$ is an invariant subspace of $\{B_i \}_{i \in \set I}$, 
and $range(P)$ is an invariant subspace of $\{A_i \}_{i \in \set I}$.  For the case of complex matrices with  $A_i = B_i$ for all~$i$,  
any eigenspace of $P$ is an invariant subspace of $\{A_i \}_{i \in \set I}$.  The following ``coupled" versions of these facts are used to prove coupled versions of 
Schur's lemma for \set A, \set B.  In the coupled versions, the $X_i$'s play the role of the $P$.

\begin{lemma}
\label{lemmabasicfacts}
Let $\set A = \{A_{ij} \}_{i, j \in \set I}$ and $\set B = \{B_{ij} \}_{i, j \in I}$, where $A_{ij}$ is $n_i \times n_j$ and 
 $B_{ij}$ is $m_i \times m_j$.  Let $X_i$ be $n_i \times m_i$ and suppose for all $i, j \in \set I$, we have 
   $A_{ij} X_j = X_i B_{ij}.$
  If $m_i = n_i$ for some $i$, then, for any scalar $\alpha$, define 
$ \set U_i(\alpha) =     \{ \vect v \  \big|  \ X_i \vect v = \alpha \vect v \}.$
The following hold for all $i, j \in \set I$.  
\begin{enumerate}
\item $B_{ij} (ker(X_j)) \subseteq ker(X_i)$. 
\item $
A_{ij} (range (X_j))  \subseteq range(X_i)$.
\item  If $\set A = \set B$, then $A_{ij} (\set U_j (\alpha))  \subseteq \set U_i(\alpha)$.
\end{enumerate}
\end{lemma}

\begin{proof}
For any $\vect w \in ker(X_j)$,  we have $X_i ( B_{ij} \vect w) =  A_{ij} X_j \vect w = 0$.  Hence,  $B_{ij} \vect w \in ker(X_i)$.  
This proves  1.
 
For $\vect w \in \set W$, we have 
$A_{ij} (X_j  \vect w) = X_i (B_{ij} \vect w) \in range(X_i)$, proving  2.

Finally, suppose $\set A = \set B$.  Then $A_{ij} X_j = X_i A_{ij}$ for all $i, j$.  Let $\vect v \in \set U_j(\alpha)$.  Then 
$X_i(A_{ij} \vect v) = A_{ij} X_j \vect v = \alpha (A_{ij} \vect v)$, showing $A_{ij} \vect v \in \set U_i(\alpha)$.  
\end{proof}

If $m_i = n_i$ and $\alpha$ is an eigenvalue of $X_i$, with $\alpha \in \field F$, then $\set U_i(\alpha)$ is the corresponding eigenspace.  
If $\alpha$ is not an eigenvalue of $X_i$, then $\set U_i(\alpha)$ is the zero subspace.

  We now state a version of Schur's Lemma for families that are irreducible in the coupled sense.  The proofs simply extend 
  the argument used to prove the usual Schur Lemma.   

\begin{theorem} 
\label{maintheorem}
Let $\set A = \{A_{ij} \}_{i, j \in \set I}$ and $\set B = \{B_{ij} \}_{i, j \in I}$, where $A_{ij}$ is $n_i \times n_j$ and 
$B_{ij}$ is $m_i \times m_j$.  Let $X_i$ be $n_i \times m_i$ and suppose for all $i, j \in \set I$, we have 
   $A_{ij} X_j = X_i B_{ij}$.
\begin{enumerate}
\item Suppose both \set A and \set B are irreducible in the coupled sense.  
Then either $X_i = 0$ for all $i$, or $X_i$ is nonsingular for all $i$.  In the latter case, $m_i = n_i$ for all $i$.  
If $\set A = \set B$, and \set A is a family of complex matrices, then there is a scalar $\alpha$ such that $X_i = \alpha I_{n_i}$ for all $i$.  
\item Suppose neither \set A nor \set B is properly reducible in the coupled sense.  
Then for each $i$, either  $X_i = 0$ or $X_i$ is nonsingular.  If $X_i$ is nonzero we must have $m_i = n_i$.  
If $\set A = \set B$ and consists of complex matrices, then any nonzero $X_i$ is a scalar multiple of $I_{n_i}$. 
\end{enumerate}
\end{theorem}

\begin{proof}
For part 1, assume \set A and \set B are both coupled irreducible.  
Consider the subspaces $ker(X_i), i \in \set I$.  Since \set B is irreducible in the coupled sense, statement~1 of Lemma~\ref{lemmabasicfacts} tells us  there are only two possibilities:  
either $ker(X_i) = \{0\}$ for all $i$, or $ker(X_i) = \set W_i$ for all $i$.  In the 
latter case, $X_i = 0$ for all $i$ and so we are done.   

Suppose now that $ker (X_i)  = \{0\}$ for all $i$.   We now use the subspaces $range(X_i), i \in \set I$.  
Since \set A is irreducible in the coupled sense, part~2 of Lemma~\ref{lemmabasicfacts}  tells us the only possibilities are 
$range(X_i) = \{0\}$ for all~$i$ or $range(X_i) = \set V_i$ for all $i$.  If $range(X_i) = \{0\}$ for all $i$, then 
$X_i = 0$ for all~$i$.  Otherwise, we have both $ker X_i = \{0\}$ and $range(X_i) = \set V_i$ for all $i$. 
Hence each $X_i$ is nonsingular and $m_i = n_i$.  

Now suppose $\set A = \set B$ and $\field F = \blabold C$.  Let $\lambda$ be an eigenvalue of $X_p$ for some fixed $p \in \set I$.   
Part~3 of  Lemma~\ref{lemmabasicfacts} tells us 
$A_{ij} ( \set U_j(\lambda) ) \subseteq \set U_i(\lambda)$ for all~$i,j$.    Since~\set A is irreducible in the coupled sense, there are then only two 
  possibilities for the subspaces $\set U_i(\lambda)$: either they are all zero, or $\set U_i= \set V_i$ for all $i$.  Since $\lambda$ was chosen to be
  an eigenvalue of $X_p$, we know $\set U_p(\lambda)$ is not zero.  Therefore, $\set U_i(\lambda) = \set V_i$ for all $i $ and hence
  $X_i = \lambda I_{n_i}$ for all~$i$.
  
  For part 2, assume neither \set A nor \set B is properly reducible in the coupled sense.  
  Consider $ker(X_i)$.  Since \set B is not properly reducible in the coupled sense, Lemma~\ref{lemmabasicfacts} tells us
$ker(X_i)$ cannot be a nonzero, proper subspace of~$\set W_i$.  Hence, for each particular $i$, either $ker(X_i ) = \{0\}$ or 
$ker(X_i) = \set W_i$.  In the latter case, $X_i = 0$.  

Suppose $ker(X_i )= \{0\}$ for some $i$.   
Since \set A is not properly reducible in the coupled sense, Lemma~\ref{lemmabasicfacts} tells us 
$range(X_i)$ is either $\{0\}$  or $\set V_i$.  If $range(X_i) = \{0\}$ then 
$X_i = 0$ .  Otherwise, we have both $ker X_i = \{0\}$ and $range(X_i) = \set V_i$, so $X_i$ is nonsingular and $m_i = n_i$.  

Now suppose $\set A = \set B$  is a family of complex matrices.  Suppose~$X_p \neq 0$ for some $p$.  Let  $\lambda_p$ be an eigenvalue of $X_p$.  Note 
$\lambda_p \neq 0$ because $X_p$ is nonsingular.  By Lemma~\ref{lemmabasicfacts}, we have
$A_{ij} ( \set U_j(\lambda_p) ) \subseteq \set U_i(\lambda_p)$ for all $i, j$.     
Since~\set A is not properly reducible in the coupled sense, each $\set U_i(\lambda_p)$ is either zero or  the full vector space $\set V_i$.  
Since $\lambda_p$ is an eigenvalue of $X_p$, the space $\set U_p(\lambda_p)$ is not zero.  Therefore, $\set U_p(\lambda_j) = \set V_p$ and 
  $X_p = \lambda_p I_{n_p}$.
 \end{proof}
 
 \begin{remark} The ordinary version of Schur's Lemma, Theorem~\ref{SchurLemma}, applies to the case where both  \set A and \set B are 
 irreducible in the sense of Definition~\ref{transformationreducible}, and $X_i = P$ for all $i$.  
 \end{remark}
 
 Note the different conclusions for the two parts of Theorem~\ref{maintheorem}.  For part~1, either {\it all} $X_i$'s are zero, or {\it all} are nonsingular.  When 
 $\set A = \set B$ and $\field F = \blabold C$, {\it all} the $X_i$'s  are the same scalar multiple of the identity matrix.  In part~2, there are more options for the $X_i$'s.  Each 
 $X_i$ is either zero or nonsingular, but some can be zero and others nonsingular.  For $\set A = \set B$ and $\field F = \blabold C$, the 
 proof for part~2 gives 
 $X_p = \lambda_p I_{n_p}$ for a particular value of~$p$; it does not show every nonzero $X_i$ equals the same scalar multiple 
 of the identity matrix.  
 
 The broader range of options for the $X_i$'s in part~2 makes sense when we consider that, at least for $\vert \set I \vert  \geq 4$, we have 
 $PropRed(\field F, \{n_i\}_{i \in \set I}) \subset Red(\field F, \{n_i\}_{i \in \set I})$, and hence 
 $Red^C(\field F, \{n_i\}_{i \in \set I}) \subset PropRed^C(\field F, \{n_i\}_{i \in \set I})$.  
 Part~2 applies to a broader set of pairs \set A, \set B than part~1.

Consider the situation in part~2 of Theorem~\ref{maintheorem}.  Suppose $X_i = 0$ and~$X_j$ is nonsingular.  
The equation $A_{ij} X_j = X_i B_{ij}$ then tells us $A_{ij} = 0$, while 
$A_{ji} X_i = X_j B_{ji}$ gives $B_{ji} = 0$.  Set 
\[\set I_0 = \{ i \in \set I \ \vert \  X_i = 0 \} \hskip .2 in {\rm and} \hskip .2 in \set I_{non}= \{ i \in \set I \ \vert \ X_i \  \rm{ is  \ nonsingular} \}.\]  
We have 
$A_{ij} = 0$ and $B_{ji} = 0$ whenever $i \in \set I_0$ and $j \in \set I_{non}$.  For example, suppose $\set I = \{1, 2, \ldots, K\}$, and, for some $0 < s < K$, 
we have 
$\set I_0 = \{1, 2, \ldots, s\}$ and  $\set I_{non} = \{s+ 1, s+2, \ldots, K\}$.  The $N \times N$ matrix $A$ then has only zero blocks in the upper right 
hand corner formed from the first $s$ rows and last $K - s$ columns.  The $M \times M$ matrix $B$ has zero blocks in the lower left hand corner 
formed by the last $K - s$ rows and first $s$ columns.  
 \[
 A = \pmatrix{A_{11} & \cdots & A_{1s} & 0 & \cdots & 0 \cr 
 \vdots &  & \vdots & \vdots & & \vdots \cr 
 A_{s1} & \cdots & A_{ss} & 0 & \cdots & 0 \cr 
 A_{(s+1)1} & \cdots & A_{(s+1)s}  & A_{(s+1)(s+1)} & \cdots & A_{(s+1)K} \cr
 \vdots & & \vdots & \vdots & & \vdots \cr
 A_{K1} & \cdots & A_{K s}  & A_{K(s+1)} & \cdots & A_{KK} }.
 \]
 Returning to the case of general \set I, one can check that   \set A is coupled reducible via the subspaces $\set U_i = \{0\}$ for $i \in \set I_0$, and 
$\set U_i =  \set V_i$   for $i \in \set I_{non}$.   
 The family~\set B is coupled reducible via the subspaces $\set U_i  = \set W_i$ when~$i \in \set I_0$,
 and $\set U_i  = \{0\}$ when~$i \in \set I_{non}$.

 
  \section{Strong reducibility and Schur's Lemma}
  \label{sectionstrongreduc}
 We now consider strongly coupled reducibility.  
 Our goal is a version of Schur's lemma for families that are not strongly reducible in the coupled sense, with a conclusion similar to 
 that of Theorem~\ref{maintheorem}: each $X_i$ is either zero or nonsingular.  
 The next example shows that for such a conclusion, we need some restrictions on $A_{ij}$ and  $B_{ij}$.

\begin{example} Let $n = m = 2$ (so $\set V = \set W=  \field F^2$), and  $K = 2$.  Put 
\[
A_{11} = A_{22}   =  \pmatrix{0 & 1 \cr 0 & 0}  \hskip .3 in 
A_{21}   =   \pmatrix{0 & 0 \cr 0 & 0}  \hskip .3 in
A_{12}   =   \pmatrix{a & b \cr c & d}
\]
\[
B_{11} = B_{22}   =   \pmatrix{0 & 1 \cr 0 & 0}   \hskip .3 in
B_{21}   =   \pmatrix{a & b \cr c & d}  \hskip .3 in
B_{12}   =   \pmatrix{0 & 0 \cr 0 & 0}. 
\]
In terms of the matrices $A, B$:  
\begin{eqnarray}
A = \pmatrix {0 & 1 & | & a & b \cr 0 & 0 & | & c & d \cr  \hline \cr 0 & 0 & | & 0 & 1 \cr 0 & 0 & | & 0 & 0}  \ \ \ \ \ 
B = \pmatrix {0 & 1 & | & 0 & 0 \cr 0 & 0 & | & 0 & 0 \cr  \hline \cr a & b & | & 0 & 1 \cr c & d & | & 0 & 0}. \nonumber
\end{eqnarray}
Let $\set U$ be the subspace spanned by $\vect e_1 = \pmatrix{1 \cr 0}$.
One may easily check that~\set A is properly reducible in the coupled sense  with $\set U_1 = \set U$ and $\set U_2 = \{0\}$, 
while~\set B is properly reducible in the coupled sense with $\set U_1 = \{0\}$ and $\set U_2 = \set U$.  However, if $c \neq 0$, then neither \set A nor \set B is 
strongly reducible in the coupled sense.  The reason is that \set U is the only nonzero, proper invariant subspace for the diagonal blocks,
$\pmatrix{0 & 1 \cr 0 & 0}$,
of $A$ and $B$, so $\set U_1 = \set U_2  = \set U$ is the only possible choice for nonzero, proper subspaces $\set U_1$ and $\set U_2$.
If $c \neq 0$, then $A_{12} (\set U) \not\subset \set U$, and $B_{21} (\set U) \not\subset \set U$.
So if $c\neq 0$, 
neither \set A nor \set B is strongly reducible in the coupled sense.  
Set $X_1 =  \pmatrix{0 & 1 \cr 0 & 0}, \ \ \ X_2 = \pmatrix{0 & 0 \cr 0 & 0}$ and $X = X_1 \oplus X_2$.  One may check that $AX =XB = 0$ and hence
$A_{ij} X_j = X_i B_{ij}$ for $i, j = 1, 2$.  The point is that  the matrix $X_1$ is neither zero nor nonsingular.  
\end{example}

Our theorem for coupled pairs \set A, \set B that are not strongly reducible will be for the case when $n_i = n$ and $m_i = m$ for all $i$.
It will have a hypothesis about graphs related to \set A and \set B; roughly speaking, this hypothesis will tell us there are ``enough" 
nonsingular $A_{ij}$'s and $B_{ij}$'s.    
Although our main result assumes \set A is a family of $n \times n$ matrices and \set B is a family of $m \times m$ matrices, we 
define the graphs for families with matrices of any size.  

 Recall that a matrix is said to have 
{\em full column rank} if the columns are linearly independent; thus, the rank of the matrix equals the number of columns.  
If $A$ is a $p \times q$ matrix with full column rank, and $\set U$ is a subspace of $\field F^q$, then $A (\set U)$ has the same
dimension as $\set U$.   

Consider $\set A = \{A_{ij} \}_{i, j \in I}$, and subspaces $\{\set U_i\}_{i \in \set I}$, satisfying $A_{ij}(\set U_j) \subseteq \set U_i$ for all $i, j$.     
 Let $d_i$ be the dimension of $\set U_i$.   
  If $A_{ij}$ has full column rank, then $A_{ij} (\set U_j) \subseteq \set U_i$ tells us $d_j \leq d_i$.  If $A_{ji}$ also 
  has full column rank, then we also have $d_i \leq d_j$, and hence $d_i = d_j$. When
  $A_{ij}$ and $A_{ji}$ both have full column rank, $n_j \leq n_i$ and $n_i \leq n_j$, so $n_i = n_j$;  
  hence,  $A_{ij}$ and $A_{ji}$ are actually square, nonsingular matrices.
  If all of the  $A_{ij}$'s  have full column rank,  then all of the $n_i$'s have the same value, $n$, and all of the subspaces~$\set U_j$  have the same dimension, $d$.
  However,  we need not assume {\it all}  of the matrices~$A_{ij}$ are nonsingular in order to show the $\set U_j$'s all have the same dimension.  
To explore this further, we introduce a directed graph  in which directed edges correspond to the $A_{ij}$'s of full column rank.  
 
 \begin{definition} Let $\set A = \{A_{ij} \}_{i, j \in I}$, with $A_{ij}$ of size $n_i \times n_j$.
 The {\em directed graph (digraph) of \set A}, denoted $\set D({\set A})$, is the graph on vertices $\{v_i\}_{i \in \set I}$, 
 such that there is a directed edge $(v_i, v_j)$ from $v_i$ to $v_j$ if and only if 
 $A_{ij}$ has full column rank.  
 \end{definition}
 
 For a finite index set, $\set I = \{1, \ldots, K\}$, there are $K$ vertices.  
 If $n_i = 1$  for all $i$, our $\set D({\set A})$  is just the usual directed graph associated with a $K \times K$ matrix.
 
More generally, there is a vertex for each $i \in \set I$, so there could be infinitely many vertices.  We use the same 
definition for directed walk as for graphs with a finite number of vertices.  A {\em directed walk} is a finite sequence of vertices, 
$v_{i_1}, v_{i_2}, \ldots, v_{i_p}$, such that $(v_{i_{j}}, v_{i_{(j+1)}})$ is a directed edge for $1 \leq j \leq (p-1)$.  In this case, we  write 
$v_{i_1}\to  v_{i_2} \to  \cdots \to v_{i_p}$.  
Vertices $v$ and  $w$ in a directed graph \set D  are said to be {\it strongly connected} if there is a directed walk from 
 $v$ to $w$ and a directed walk from $w$ to $v$.   We say~\set D is
  strongly connected if each pair of vertices of \set D is strongly connected. 
 \begin{proposition}
 \label{subspacedimfact} 
 Let  $\set A = \{A_{ij} \}_{i, j \in \set I}$ and suppose the subspaces $\{U_i\}_{i \in \set I}$ satisfy  $A_{ij} (\set U_j) \subseteq \set U_i$,
for all $i, j$. 
 Then the following hold.  
 \begin{enumerate}
 \item If there is a directed walk from $v_i$ to $v_j$ in $\set D(\set A)$, then $n_j \leq n_i$, and 
 $dim(\set U_j) \leq dim(\set U_i)$.  
 
 \item If the vertices $v_i$ and $v_j$ are strongly connected in $\set D(\set A)$, then $n_i = n_j$, and 
$dim(\set U_j) = dim(\set U_i)$
 
 \item If $\set D(\set A)$ is strongly connected, all of the~$n_i$'s are equal, and all of the subspaces $\set U_i$ have the same dimension. 
 \end{enumerate}
\end{proposition}

\begin{proof}
Let $d_i = dim(\set U_i)$.  If  $(v_i, v_j)$ is a directed edge of $\set D(\set A)$, then $A_{ij}$ has full column rank, so, as we have already observed,  
$n_j \leq n_i$ and $d_j \leq d_i$.   

More generally,  suppose
$  v_i = v_{i_1} \to v_{i_2} \to \cdots \to v_{i_p} = v_j$
 is a directed walk from $v_i$ to $v_j$ in $\set D(\set A)$.  Working from right to left, we have 
  \[n_j =  n_{i_p} \leq n_{i_{p-1}} \leq \cdots  \leq n_{i_2}  \leq n_{i_1} = n_i \]
  and
  \[d_j =  d_{i_p} \leq d_{i_{p-1}} \leq \cdots  \leq d_{i_2}  \leq d_{i_1} = d_i.  \]
  Hence,  $n_j \leq n_i$ and $d_j \leq d_i$.  
  
  If $v_i$ and $v_j$ are strongly connected, there  is a directed walk from $v_i$ to $v_j$ and a directed walk from $v_j$ to $v_i$.
  So $n_i = n_j$ and $d_i = d_j$. 
  
   If $\set D(\set A)$ is strongly connected, then, for all $i, j$, we have $d_i = d_j$  and $n_i = n_j$, so all of the subspaces $\set U_j$ have the 
  same dimension and all of the~$n_i$'s have the same value.  
\end{proof}

  As an example, suppose $\set I = \{1, \ldots, K\}$ and $A_{12}, A_{23},  \ldots A_{K-1, K}, A_{K1}$ all have full column rank.  Then $\set D( \set A )$ contains the directed cycle
   \[v_1 \to v_2  \to \cdots \to v_{K-1} \to v_K \to v_1,\]
   and is  strongly connected.  
   
   If  $\set D(\set A)$ is not strongly connected,  the strong components  
   identify sets of~$n_i$'s which must be equal, and sets of subspaces $\set U_i$  which must have the same dimension.  
   For each strong component, \set C,  of \set D(\set A), all~$n_i$'s corresponding to vertices of $\set C$ must be equal, 
   and all subspaces $\set U_i$ corresponding to 
   vertices of~$\set C$ must have the same dimension.  For a finite $\set I$,  
   we can use the strong components to put the $N \times N$ matrix $A$ into a block triangular 
   form in which none  of the $A_{ij}$'s below the diagonal blocks has full column rank.     (See~\cite{BR91},  section 3.2.)
   
For the proofs of coupled versions of Schur's Lemma, the subspaces $\set U_i$  of interest are the kernels and ranges of the matrices $\{X_i \}_{i \in \set I}$.

\begin{proposition}
\label{dimensionfacts}
Let  $\set A = \{A_{ij} \}_{i, j \in \set I}$   and    $\set B = \{B_{ij} \}_{i, j \in \set I}$.  Let $X_i$ be $n_i \times m_i$, and  
 suppose $A_{ij} X_j = X_i B_{ij}$ for all $i, j \in \set I$.
Then the following hold. 

\begin{enumerate}
\item  If $v_i$ and $v_j$ are strongly connected in \set D(\set A), then $range(X_i) $ and $range (X_j) $ have the same 
dimension, i.e.,  $X_i$ and $X_j$ have the same rank. 

\item If $v_i$ and $v_j$ are strongly connected in \set D(\set B), then $ker(X_i)$ and $ker(X_j) $ have the same 
dimension, i.e., $X_i$ and $X_j$ have the same nullity. 

\item If \set D(\set A) is strongly connected, all of the $n_i$'s have the same value,~$n$, and all of the 
$X_i$'s have the same rank. 

\item  If \set D(\set B) is strongly connected, all of the $m_i$'s have the same value,~$m$, and all of the $X_i$'s have the same nullity, $d$. 

\item If $v_i$ and $v_j$ are strongly connected in \set D(\set B), then $X_i$ and $X_j$ have the same rank. 

\item  If \set D(\set B) is strongly connected, all of the $X_i$'s have the same rank.
\end{enumerate}
\end{proposition} 

\begin{proof}
The first four parts follow from Lemma~\ref{lemmabasicfacts} and Proposition~\ref{subspacedimfact}. 
For part~5, suppose $v_i$ and $v_j$ are strongly connected in \set D(\set B).  Then $m_i = m_j$, so the matrices $X_i$ and $X_j$ have the same number of 
columns.  From part~2, we know $X_i$ and $X_j$ have the same nullity.  The rank plus nullity theorem then tells us 
$X_i$ and $X_j$ have the same rank.  Part 6 is an immediate consequence of part 5.  
\end{proof}

We now have a version of Schur's lemma for families \set A, \set B when neither is strongly reducible in the coupled sense.  

\begin{theorem}
\label{Schurlemmaversionthree}
Assume neither  $\set A = \{A_{ij} \}_{i, j \in I}$  nor     $\set B = \{B_{ij} \}_{i, j \in I}$  is strongly reducible in the coupled sense.  Assume 
also that \set D(\set A) and \set D(\set B) are strongly connected.  
Let $X_i$ be $n \times m$ for all $i \in \set I$, and suppose 
 $A_{ij} X_j = X_i B_{ij}$ for all $i, j \in \set I$.  
Then either $X_i = 0$ for all $i$, or $X_i$ is nonsingular for all $i$.  In the latter case we must have $m = n$.  
If $\set A = \set B$ and is a family of complex matrices, then there is some scalar $\alpha$ such that $X_i = \alpha I_n$ for all $i$.  
\end{theorem}

\begin{proof}
Note first that since \set D(\set A) and \set D(\set B) are both strongly connected, the~$A_{ij}$'s are all square matrices of the same size, $n$, and 
the $B_{ij}$'s are all square matrices of the same size, $m$.  

By Proposition~\ref{dimensionfacts},  the subspaces $ker(X_i)$, for $i \in \set I$, all have the same 
dimension,~$d$.   Since \set B is not strongly reducible in the coupled sense,  either $d = 0$ or $d = m$.  
If $d = m$, then $X_i = 0$ for all $i$ and we are done. 

Assume then that $d = 0$.  Proposition~\ref{dimensionfacts} tells us the subspaces 
$range(X_i)$ all have the same dimension, $r$.  Since \set A is not strongly reducible in the coupled sense either $r = 0$ or $r = n$.  If $r =0$, then 
$X_i = 0$ for all $j$.  If $r = n$, then, since we also have $d = 0$, the $X_i$'s are nonsingular; we then have~$m = n$.

If $\set A  = \set B$, 
we have $A_{ij} X_j = X_i A_{ij}$ 
for all $i, j$.  Fix $p$ and let $\lambda$ be an eigenvalue of $X_p$ with corresponding eigenspace $\set U_p ({\lambda})$; note the 
subspace $\set U_p ({\lambda})$ is nonzero, because $\lambda$ is an eigenvalue of $X_p$.  From part 3 of Lemma~\ref{lemmabasicfacts}, we have   
 $A_{ij} (\set U_j(\lambda) )  \subseteq \set U_i(\lambda)$ for all $i, j$.  
Since $\set D(\set A)$ is strongly connected, Proposition~\ref{subspacedimfact} tells us 
 the spaces $\set U_i(\lambda)$ all have the same dimension; call it~$f$.  Since $\set U_p ({\lambda})$ is nonzero, we know $f > 0$.
 Hence, since \set A is not strongly reducible in the coupled sense, 
 we must have $f = n$, and  $X_i  = \lambda I_n$ for all~$i$.   
\end{proof}

\begin{remark} Earlier work~\cite{LaJu15} gives a proof, using block matrix computation, for the case where all $A_{ij}$'s and $B_{ij}$'s are assumed to be 
nonsingular.  
\end{remark}

The proof of Theorem~\ref{Schurlemmaversionthree} uses the assumption that both \set D(\set A) and \set D(\set B) are strongly 
connected in two ways:  to establish that $n_i = n$ and $m_i = m$ for all $i$, and to show that the relevant subspaces 
(kernels and ranges of the $X_i$'s) have the same dimension.  We now develop another version of  Theorem~\ref{Schurlemmaversionthree}, in which
we weaken the hypothesis about the graphs, but then need to explicitly assume that $n_i = n$ and $m_i = m$ for all $i$.
The key point for this second version is that $X_i$ and $X_j$ have the same rank whenever $v_i$ and $v_j$ are strongly connected in 
{\em either} of the digraphs \set D(\set A) or \set D(\set B).  

We  use $\set A$ and $\set B$ to define an undirected graph, 
$\set G(\set A, \set B)$,  as follows.  
\begin{definition}
The undirected graph, $\set G(\set A, \set B)$, is the graph on vertices $\{v_i\}_{i \in \set I}$, such that 
 $\{ v_i, v_j\}$ is an (undirected) edge of 
$\set G(\set A, \set B)$ if and only if the vertices $v_i$ and $v_j$ are either strongly connected in $\set D(\set A)$, or in $\set D(\set B)$ (or both).
We call this the {\it linked graph of  $\set A$ and $\set B$}.
\end{definition}

\begin{proposition}
\label{dimensionfactsgen}
Let  $\set A = \{A_{ij} \}_{i, j \in \set I}$   and    $\set B = \{B_{ij} \}_{i, j \in \set I}$.   Let $X_i$ be $n_i \times m_i$  and  
 suppose $A_{ij} X_j = X_i B_{ij}$ for all $i, j \in \set I$.
 If $v_i$ and $v_j$ are connected in \set G(\set A, \set B) then $X_i$ and $X_j$ have the same rank.  
 If \set G(\set A, \set B) is connected, then 
 all of the matrices $X_i$ have the same rank.
\end{proposition} 

\begin{proof}
Suppose  $v_i$ and $v_j$ are connected in $\set G(\set A, \set B)$.  Then there is a sequence of vertices, 
$v_i = v_{i_1}, v_{i_2}, v_{i_3}, \ldots, v_{i_{p-1}}, v_{i_p} = v_j$,
 such that $\{v_{i_k}, v_{i_{k+1}}\}$ is an edge of 
$\set G(\set A, \set B)$ for $k = 1, \dots, p-1$.  This means $v_{i_k}$ and $v_{i_{k+1}}$ are either strongly connected in $\set D(\set A)$
or strongly connected in $\set D(\set B)$, or both.  Therefore, $rank(X_{i_k})  = rank(X_{i_{k+1}})$ for $k = 1, \ldots, p-1$, 
and $rank(X_i)  = rank(X_j ) $.  
\end{proof}

If either $\set D(\set A)$ or $\set D(\set B)$ is strongly connected, then $\set G(\set A, \set B)$ will be connected.  
However, $\set G(\set A, \set B)$  can be a connected graph even if neither of the digraphs 
$\set D(\set A)$ or $\set D(\set B)$ is strongly connected.  For example, suppose $K = 3$ and 
\[
A = \pmatrix{0 & A_{12} & 0 \cr A_{21} & 0 & 0 \cr 0 & 0 & 0}  \hskip .5 in 
B = \pmatrix{0 & 0 & B_{13}  \cr 0 & 0 & 0 \cr B_{31}  & 0 & 0},
\]
where $A_{12}, A_{21}, B_{13}$ and $B_{31}$ are all nonsingular.  Neither $\set D(\set A)$ nor $\set D(\set B)$ is connected, 
but $\set G(\set A, \set B)$ is connected. (See Figure~\ref{graphpicture}.)

\begin{figure}
\begin{picture}(200, 50)(-250, 10)
\put(-230, 40){$v_1$}
\put(-165, 40){$v_2$}
\put(-170, 50){\circle*{3}}
\put(-220, 50){\circle*{3}}
\put(-220, 50){\vector(1,0){40}}
\put(-180, 50){\vector(-1,0){25}}
\put (-190, 50){\line(1,0){20}}
\put(-195, 100){\circle*{3}}
\put(-195, 110){$v_3$}
\put(-210, 20){$\set D(\set A)$}
\put(-110, 40){$v_1$}
\put(-45, 40){$v_2$}
\put(-50, 50){\circle*{3}}
\put(-100, 50){\circle*{3}}
\put(-100, 50){\vector(2,3){25}}
\put(-68, 97){\vector(-2,-3){25}}
\put(-68, 97){\circle*{3}}
\put(-75, 110){$v_3$}
\put(-90, 20){$\set D(\set B)$}
\put(10, 40){$v_1$}
\put(65, 40){$v_2$}
\put(70, 50){\circle*{3}}
\put(10, 50){\circle*{3}}
\put(10, 50){\line(2,3){35}}
\put (10, 50){\line(1,0){60}}
\put(45, 104){\circle*{3}}
\put(45, 110){$v_3$}
\put(30, 20){$\set G(\set A, \set B)$}
\end{picture}
\caption{\set D(\set A), \set D(\set B) and \set G(\set A, \set B)}
\label{graphpicture}
 \end{figure}

As an example, suppose $\set I_1$ and $\set I_2$ are nonempty, disjoint subsets of \set I such that $\set I = \set I_1 \cup \set I_2$.  
Partition the vertices of \set G(\set A, \set B) into two sets 
corresponding to $\set I_1$ and $\set I_2$, setting 
\[\set S = \{ v_i \ \vert \ i \in \set I_1 \} \hskip .2 in {\rm and} \hskip .2 in \set T = \{ v_i \ \vert \ i \in \set I_2 \}.\]
Suppose $rank(A_{ij}) < n_j$ and $rank(B_{ij}) < m_j$, whenever $i \in \set I_1$ and $j \in \set I_2$. 
Then neither \set D(\set A) nor  \set D(\set B)  has any directed edges from vertices in $\set S$ to vertices in \set T.
The linked graph   $\set G(\set A, \set B)$  then has 
no edges from vertices in~\set S to vertices in \set T  and hence is not connected.   

Now suppose that, whenever $i \in \set I_1$ and $j \in \set I_2$, we have $rank(A_{ij}) < n_j$ and $rank(B_{ji}) < m_i$, (note the reversal of subscripts on $B_{ji}$).
In this case,  $\set D(\set A)$ has no directed edges from vertices in $\set S$ to vertices in \set T, 
while $\set D(\set B)$ has no directed edges from vertices in \set T to vertices in \set S.  Consequently, if $v \in \set S$ and 
$w \in \set T$, then the pair $v, w$ is not strongly connected in either  \set D(\set A) or \set D(\set B).  Hence, $\set G(\set A, \set B)$ has no edges between vertices in 
\set S and vertices in~\set T; thus $\set G(\set A, \set B)$ is not connected.

The following variation of Theorem~\ref{Schurlemmaversionthree} uses this linked graph, \set G(\set A, \set B).  

\begin{theorem}
\label{Schurlemmaversionfour}
Assume neither  $\set A = \{A_{ij} \}_{i, j \in I}$  nor $\set B = \{B_{ij} \}_{i, j \in I}$  is strongly reducible in the coupled sense.  Assume 
also that $n_i = n$ and $m_i = m$ for all $i$, and that $\set G(\set A, \set B)$ is connected.
Let $X_i$ be $n \times m$, and suppose 
 $A_{ij} X_j = X_i B_{ij}$ for all $i, j \in \set I$.  
Then either $X_i = 0$ for all $i$, or $X_i$ is nonsingular for all $i$.  In the latter case we must have $m = n$.  
If $\set A = \set B$ and is a family of complex matrices, then there is some scalar $\alpha$ such that $X_i = \alpha I_n$ for all $i$.  
\end{theorem}

\begin{proof}
By Proposition~\ref{dimensionfactsgen}, the subspaces $range(X_i)$, for $i$ in \set I, all have the same dimension,~$r$.
Since all of the $X_i$'s have the same number of columns, the rank plus nullity theorem tells us the 
subspaces $ker(X_i)$, for $i \in \set I$, must also all have the same 
dimension,~$d$.
The remainder of the proof is the same as that for Theorem~\ref{Schurlemmaversionthree}.
\end{proof}

Comparing Theorems~\ref{maintheorem},~\ref{Schurlemmaversionthree}, and~\ref{Schurlemmaversionfour}, 
the simplest version is part~1 of Theorem~\ref{maintheorem}.  It is the closest to the usual Schur's Lemma.  However,  
the hypothesis that \set A, \set B be irreducible in the coupled sense is more restrictive than the hypothesis 
of part~2 of Theorem~\ref{maintheorem}. 
The conclusion of part~2 has more options for the~$X_i$'s than part~1.
Theorems~\ref{Schurlemmaversionthree} and~\ref{Schurlemmaversionfour} apply to the larger class of pairs,~\set A,~\set B, which are not strongly reducible
in the coupled sense, 
but have additional restrictions about the connectivity of the graphs \set D(\set A), \set D(\set B), and  
\set G(\set A, \set B) and the equality of the $n_i$'s and $m_i$'s.


\section{Normality and coupled normality}
\label{normalcoupled}

We now consider a refinement of Schur's Lemma for irreducible sets of normal matrices.  This is closely related to Lemma A.4 of~\cite{MKKK10}.  We 
obtain corresponding results for sets \set A, \set B satisfying
 a ``coupled normality" condition.  
  For this section we work over the  field of complex numbers.  
We use * to denote the transpose conjugate of a matrix.  If \set U is a subspace of \set V, we use $\set U^{\perp}$ for the orthogonal 
complement of $\set U$.  We will need the following facts.

\begin{proposition}
\label{propnormalfacts}
Let $A$ be a normal matrix; let $S$ be nonsingular and let $B = S^{-1}AS$.  Then the following are equivalent.

\begin{enumerate}

\item The matrix $B$ is normal. 

\item $S^{-1}A^* S = B^*. $

\item The matrix $SS^*$ commutes with $A$. 

\item The matrix $SS^*$ commutes with $A^*$.  

\item The matrix $S^*S$ commutes with $B$. 

\item The matrix $S^*S$ commutes with $B^*$.

\end{enumerate}

\end{proposition}

\begin{proof}
The equivalence of 2, 3 and 4 is easily shown.  Using $B = S^{-1}AS$, 
\begin{eqnarray}
S^{-1} A^* S = B^* &  \iff  &  S^{-1}A^* S = S^* A^* S^{-*} \cr  
& \iff & A^* SS^* =   SS^* A^*  \cr 
& \iff &  SS^* A =   A SS^* , \nonumber
\end{eqnarray} 
where the third line comes from taking the transpose conjugate of the equation in the second line.   
A similar calculation,  starting with $A = S B S^{-1}$, shows  2, 5 and 6 are equivalent: 
\begin{eqnarray}
S B^* S^{-1}  = A^* &  \iff  &  S B^* S^{-1} = S^{-*} B^* S^* \cr 
& \iff & S^* S B^*=   B^* S^* S \cr 
& \iff &  B S^*S =    S^*S B . \nonumber
\end{eqnarray} 
The fact that 2 implies 1 is also easy.  If $S^{-1} A^* S = B^* $,  use $AA^* = A^*A$ to get
\[
B B^*  =   (S^{-1} AS) (S^{-1} A^* S) = (S^{-1} A^*S)(S^{-1} AS) = B^* B.
\]

The only part needing any work at all is to show  1 implies 2.  Let $\scalarlist{\lambda}{n}$ be the eigenvalues of $A$ and let 
$D$ be the diagonal matrix with diagonal entries $\scalarlist{\lambda}{n}$.  Since $A$ is normal,  $A = U^*DU$ for some 
unitary matrix~$U$, and $A^* = U^* \overline{D} U$, where the bar denotes complex conjugation.  Note 
$\scalarlist{\overline{\lambda}}{n}$ are the eigenvalues of $A^*$.  Let $p(x)$ be a polynomial such that 
$p(\lambda_i) = \overline{\lambda}_i$ for each eigenvalue $\lambda_i$.  Then $\overline{D} = p(D)$, and
\[A^* = U^* p(D)U = p(U^* DU) = p(A).\]
Since $B$ is similar to $A$, the matrix~$B$ also has eigenvalues 
$\scalarlist{\lambda}{n}$ and~$B^*$ has eigenvalues $\scalarlist{\overline{\lambda}}{n}$. If $B$ is normal, $B = V^*DV$ for some unitary matrix $V$. 
Hence, 
\[B^* = V^*\overline{D} V = V^*p(D)V = p(V^*DV) = p(B).\]
 But 
$p(B) = p(S^{-1}AS) = S^{-1} p(A) S = S^{-1} A^* S$, so $B^* = S^{-1} A^* S$.
\end{proof}  
%
%

%
This gives an easy proof of the following.  
\begin{theorem}
\label{theoremirrednormal}
Suppose  $\{A_i\}_{i \in \set I}$  and $\{B_i\}_{i \in \set I}$ are  irreducible families of normal matrices, and $S$ is a nonsingular matrix such that 
$S^{-1} A_iS = B_i$ for all~$i \in \set I$.   Then $S$ is a scalar multiple of a unitary matrix. 
\end{theorem}
\begin{proof}
By the preceding proposition, $SS^*$ commutes with each $A_i$.  Since $\{A_i\}_{i \in \set I}$ is an irreducible family, $SS^*$ must be a scalar matrix.  
Since $S$ is nonsingular, the Hermitian matrix $SS^*$ is positive definite; hence $SS^* = \alpha I$ where $\alpha$ is a positive real number. 
Set $U = {1 \over {\sqrt \alpha}} S$.  Then $ U U^* = { 1 \over {\alpha}} SS^* = I$.  So $U$ is unitary and $S = \sqrt {\alpha} U$.  
\end{proof}
\begin{remark}
This argument is essentially the proof of 
Lemma A.4 of~\cite{MKKK10}, which says that if two irreducible representations of a
$*$-algebra of square matrices are equivalent, then they are similar via a unitary similarity.  Let \set S be the algebra generated by $\{A_i\}_{i \in \set I}$ and let 
\set T be the algebra generated by $\{B_i\}_{i \in \set I}$. For any normal matrix, $N$, the matrix $N^*$ is a polynomial in $N$, so the algebras \set S and 
\set T are $*$-algebras, (which means that whenever $A$ is in the algebra, so is $A^*$).  Let $S$ be a nonsingular matrix such that 
$S^{-1} A_iS = B_i$ for all $i$.  Proposition~\ref{propnormalfacts} tells us $S^{-1} A_i^*S = B_i^*$ for all $i$, so $S$ may be extended to an 
isomorphism of the $*$-algebras \set S and \set T in the usual way.  
\end{remark}

We now introduce the idea of coupled normality. 

\begin{definition}
The family $\set A = \{ A_{ij}\}_{i, j \in \set I}$ is {\it normal in the coupled sense} if  for all $i, j \in \set I$
we have
$A_{ij}^* A_{ij}  = A_{ji} A_{ji}^*$.    
\end{definition}

If \set A is normal in the coupled sense, setting $i=j$ gives $A_{ii}^*A_{ii} = A_{ii} A_{ii}^*$, so~$A_{ii}$ is 
normal for all $i$.   Note also that if $A_{ji} = A_{ij}^*$ for all $i, j$, then \set A is coupled normal.  When $\set I = \{1, \ldots, K\}$, the condition
$A_{ji} = A_{ij}^*$ for all $i, j$ holds when $A$ is a Hermitian matrix.  In the JISA model, $A$ is a covariance matrix, and hence is a real, symmetric 
matrix, so it is Hermitian.  

Recall that, for any matrix $G$, the four matrices $G, G^*, GG^*$, and $G^*G$ all have the same rank.  Hence, when \set A is 
normal in the coupled sense, the matrices $A_{ij}$ and $A_{ji}$ have the same rank.  In particular, note that $A_{ij}$ is nonsingular if and only 
if  $A_{ji}$ is nonsingular.

Let $C$ be a $q \times p$  matrix, let  $D$ be a $p \times q$ matrix, and let $M$ be the $(p+q) \times (p+q)$ matrix 
\[M = \pmatrix{ 0 & D \cr C & 0},\]
where the zero blocks are $p \times p$ and $q \times q$.  Then
\[MM^*  =  \pmatrix{ DD^* & 0 \cr  0 & C C^*} \hskip .2 in  {\rm and} \hskip .2 in
M^*M  =  \pmatrix{ C^*C & 0 \cr  0 & D^* D}.\]
Hence, $M$ is normal if and only if $C^*C = DD^*$ and $D^* D = CC^* $.  The connection with coupled normality is this:   if we set 
$M_{ij}  =  \pmatrix{0 & A_{ij} \cr A_{ji} & 0}$, 
then $\set A = \{ A_{ij}\}_{i\in \set I}$ is normal in the coupled sense if and only if $M_{ij}$ is normal for all $i, j \in \set I$.  

Suppose \set A is normal in the coupled sense and the subspaces $\{ \set U_i\}_{i \in \set I}$ satisfy
 $A_{ij} (\set U_j )\subseteq \set U_i$ for all $i, j$.   Let $d_i$ be the dimension of $\set U_i$. 
 We use the fact that $M_{ij}$ is normal to show that 
$A_{ij} (\set U_j^{\perp})  \subseteq \set U_i^{\perp} $ for all~$i, j$.    

\begin{proposition}
\label{propcouplednormal}
Let $C, D$  be matrices of sizes $q \times p$ and $p \times q$, respectively, such that $C^*C = DD^*$ and $D^* D = CC^*$. 
Suppose there are subspaces  $\set U$ of~$\blabold C^p$, and \set W of $\blabold C^q$, such that 
$C(\set U) \subseteq \set W$ and $D(\set W) \subseteq \set U$.  Then $C(\set U^{\perp}) \subseteq \set W^{\perp}$ and 
$D(\set W^{\perp}) \subseteq \set U^{\perp}$. 
\end{proposition}

\begin{proof}
Let $M = \pmatrix{ 0 & D \cr C & 0}$.  For any $\vect x \in \blabold C^p$ and $\vect y  \in \blabold C^q$, 
\[
M \pmatrix{ \vect x \cr \vect y} = \pmatrix{D \vect y \cr C \vect x}.
\]
If $\vect x \in \set U$ and $\vect y \in \set W$, then $D \vect y \in \set U$ and $C \vect x \in \set W$.
So  $\set U \oplus \set W$, (which is a subspace of $\blabold C^p \oplus  \blabold C^q$), 
is invariant under~$M$.  Since $M$ is normal, the orthogonal complement of $\set U \oplus \set W$  in  $\blabold C^p \oplus  \blabold C^q$ must also be 
invariant under~$M$.  Hence, $\set U^{\perp} \oplus \set W^{\perp}$ is invariant under $M$.  This means that, for $\vect x \in \set U^{\perp}$
and $\vect y \in \set W^{\perp}$, we have $D \vect y \in \set U^{\perp}$ and $C \vect x \in \set W^{\perp}$.  
So $C(\set U^{\perp}) \subseteq \set W^{\perp}$ and $D(\set W^{\perp}) \subseteq \set U^{\perp}$.
\end{proof}
Apply Proposition~\ref{propcouplednormal} to the normal matrix $M_{ij}  =  \pmatrix{0 & A_{ij} \cr A_{ji} & 0}$, to get 
$A_{ij} (\set U_j ^{\perp} ) \subseteq \set U_i^{\perp}$ for all $i, j$. 
Hence, if \set A is normal in the coupled sense, and is reducible in the coupled sense, then it is fully reducible in the coupled sense,
because we can form $T_j$ using a basis for $\set U_j$ for the first $d_j$ columns and a basis for $\set U_j^{\perp}$ for the remaining $n -d_j$ columns.
If we use orthonormal bases for $\set U_j$ and $\set U_j^{\perp}$, then $T_j$ will be unitary. 
Hence \set A is fully reducible in the coupled sense with a coupled unitary similarity.

We will give three versions  of Theorem~\ref{theoremirrednormal} for  \set A, \set B which are normal in the coupled sense, corresponding to the three 
types of reducibility.    
The proofs depend on the following proposition.  The first two statements are a  ``coupled" version of Proposition~\ref{propnormalfacts}.  Part~4 uses the digraphs 
\set D(\set A) and~\set D(\set B).

\begin{proposition}
\label{propcouplednormalfacts}
Assume the families  $\set A = \{A_{ij} \}_{i, j \in \set I}$   and    $\set B = \{B_{ij} \}_{i, j \in \set I}$,   where $A_{ij}$ and $B_{ij}$ are 
complex matrices,  are normal in the coupled sense.
Suppose 
 $A_{ij} S_j = S_i B_{ij}$ for all $i, j$, where $S_i$ is $n_i \times m_i$.    
 For any $i \in \set I$, and any scalar~$\alpha$, define 
 \[\set U_i(\alpha) = \{ \vect v \ \big| \ S_i S_i^* \vect v = \alpha \vect v \} \hskip .15 in {\rm and}  \hskip .15 in
  \set Y_i(\alpha) = \{ \vect w \ \big| \ S_i^* S_i \vect w = \alpha \vect w \}.\]
  Then the following hold. 
 \begin{enumerate}
 \item If $S_i$ is nonsingular then $S_i S_i^*$ commutes with~$A_{ii}$, and $S_i^* S_i$ commutes with~$B_{ii}$. 
 \item
 If $S_i$ and $S_j$ are both nonsingular,
 \[S_i S_i^* A_{ij}  =   A_{ij} S_jS_j^* \hskip .2 in   {\rm and}  \hskip .2 in
 S_i^*S_i B_{ij}  =  B_{ij} S_j^*S_j. \]
 \item
 If $S_i$ and $S_j$ are both nonsingular, 
 \[A_{ij} (\set U_j (\alpha))  \subseteq   \set U_i (\alpha) \hskip .2 in {\rm and} \hskip .2 in  
B_{ij}(\set Y_j (\alpha))  \subseteq   \set Y_i (\alpha). \]
 If $A_{ij}$ is also nonsingular, then $ dim ( \set U_i(\alpha)) =  dim ( \set U_j(\alpha))$.  
 
 If $B_{ij}$ is also nonsingular, then $ dim ( \set Y_i(\alpha)) =  dim ( \set Y_j(\alpha))$.  
 \item  Assume $S_i$ is nonsingular for all $i \in \set I$.  Then the following hold. 
 
 If $v_i$ and $v_j$ are strongly connected in \set D(\set A), then $\set U_i (\alpha)$ and $\set U_j (\alpha)$ have the same dimension. 
  
If $v_i$ and $v_j$ are strongly connected in \set D(\set B), then $\set Y_i (\alpha)$ and $\set Y_j (\alpha)$ have the same dimension.  
 \item If $\alpha \neq 0$ then $\set U_i (\alpha)$ and $\set Y_i (\alpha)$ have the same dimension. 
 \item Assume $S_i$ is nonsingular for all $i \in \set I$.  Then if $v_i$ and $v_j$ are connected in \set G(\set A, \set B), and $\alpha \neq 0$, 
 we have $dim ( \set U_i(\alpha)) =  dim ( \set U_j(\alpha))$. 
\end{enumerate}
\end{proposition}

\begin{proof}
Suppose $S_i$ is nonsingular.  Since $A_{ii}$ and $B_{ii}$ are both normal, and $S_i^{-1} A_{ii} S_i = B_{ii}$, Proposition~\ref{propnormalfacts} 
tells us that $S_i S_i^*$ commutes with~$A_{ii}$ and $S_i^* S_i$ commutes with~$B_{ii}$.    

Now suppose $i \neq j$, and $S_i$ and $S_j$ are both nonsingular.  
Set 
$M_{ij} =  \pmatrix{0 & A_{ij} \cr A_{ji} & 0}$.  Then 
\begin{eqnarray}
M_{ij} \pmatrix{S_i & 0 \cr 0 & S_j} &  =  & \pmatrix{0 & A_{ij} \cr A_{ji} & 0}   \pmatrix{S_i & 0 \cr 0 & S_j}  \cr \cr
& =  &\pmatrix{0 & A_{ij}S_j  \cr A_{ji}S_i & 0}  \cr \cr
& = & \pmatrix{0 & S_i B_{ij}  \cr S_j  B_{ji}& 0}  \cr \cr 
& = & \pmatrix{S_i & 0 \cr 0 & S_j} \pmatrix{0 & B_{ij} \cr B_{ji} & 0}. \nonumber
\end{eqnarray} 
So, 
\[ 
\pmatrix{S_i & 0 \cr 0 & S_j}^{-1} M_{ij} \pmatrix{S_i & 0 \cr 0 & S_j} = \pmatrix{0 & B_{ij} \cr B_{ji} & 0}.
\]
Since  \set A and \set B are both normal in the coupled sense, $M_{ij}$ and $\pmatrix{0 & B_{ij} \cr B_{ji} & 0}$ are both normal. 
Set $S = \pmatrix{S_i & 0 \cr 0 & S_j}$.  Proposition~\ref{propnormalfacts} tells us that  $SS^*$ commutes with $M_{ij}$.  Hence, 
\[
\pmatrix{S_i S_i^* & 0 \cr 0 & S_j S_j^*}\pmatrix{0 & A_{ij} \cr A_{ji} & 0} = \pmatrix{0 & A_{ij} \cr A_{ji} & 0}\pmatrix{S_i S_i^* & 0 \cr 0 & S_j S_j^*},
\] 
and  $S_i S_i^* A_{ij} = A_{ij} S_jS_j^*$.
Use the fact that $S^*S$ commutes with $\pmatrix{0 & B_{ij} \cr B_{ji} & 0}$ to show 
$S_i^*S_i B_{ij} = B_{ij} S_j^*S_j$
 for all $i, j$.

For part 3, assume $S_i$ and $S_j$ are nonsingular.  
Let $\vect v \in \set U_j(\alpha)$. By part~2,     
$S_i S_i^*  (A_{ij} \vect v) = A_{ij} (S_jS_j^* \vect v) = \alpha (A_{ij} \vect v)$.  This shows  
$A_{ij} (\set U_j(\alpha)) \subseteq \set U_i(\alpha)$.  If~$A_{ij}$ is nonsingular, 
$ dim (A_{ij} (\set U_j(\alpha))) = dim(\set U_j(\alpha))$, so
$ dim ( \set U_j(\alpha)) \leq  dim ( \set U_i(\alpha))$.  Since \set A is coupled normal, $A_{ji}$ is also nonsingular,  
giving the reverse inequality, so 
 ${\rm dim} ( \set U_j(\alpha)) = {\rm dim} ( \set U_i(\alpha))$. 
 The corresponding facts for \set B come from the same argument, 
using $S_i^*S_i B_{ij}  =  B_{ij} S_j^*S_j$.

For part~4, assume $v_i$ and $v_j$ are strongly connected in \set D(\set A).  
Proposition~\ref{subspacedimfact}, together with part~3, gives 
${\rm dim}(\set U_i(\alpha)) = {\rm dim}(\set U_i(\alpha))$.  The same
argument applies when $v_i$ and $v_j$ are strongly connected in \set D(\set B)

Part~5 comes from the fact that $S_j^*S_j$ and $S_j S_j^*$ have the same nonzero eigenvalues with the same multiplicities. 

For part~6, suppose $v_i$ and $v_j$ are connected in \set G(\set A, \set B).    
Then there is a sequence of vertices, 
$v_i = v_{i_1}, v_{i_2}, v_{i_3}, \ldots, v_{i_{p-1}}, v_{i_p} = v_j$,
 such that $\{v_{i_k}, v_{i_{k+1}}\}$ is an edge of 
$\set G(\set A, \set B)$ for $k = 1, \dots, p-1$.  This means $v_{i_k}$ and $v_{i_{k+1}}$ are either strongly connected in $\set D(\set A)$
or strongly connected in $\set D(\set B)$ (or both).  If $v_{i_k}$ and $v_{i_{k+1}}$ are strongly connected in $\set D(\set A)$, then 
$dim(\set U_{i_k} (\alpha))  = dim(\set U_{i_{k+1}} (\alpha))$ by part~4.  If $v_{i_k}$ and $v_{i_{k+1}}$ are strongly connected in $\set D(\set B)$, 
then part~4 tells us $dim(\set Y_{i_k} (\alpha))  = dim(\set Y_{i_{k+1}} (\alpha))$.  But, since $\alpha$ is nonzero, $\set U_i(\alpha)$ and 
$\set Y_i(\alpha)$ have the same dimension.  So, in either case, $dim(\set U_{i_k} (\alpha))  = dim(\set U_{i_{k+1}} (\alpha))$
for $1 \leq k \leq p-1$, and hence $dim(\set U_i (\alpha))  = dim(\set U_j (\alpha))$.   
\end{proof}

With these preliminaries completed, we state and prove a
version of Schur's Lemma for \set A, \set B that are normal in the coupled sense.  
The three cases correspond to the three types of coupled reducibility.  

\begin{theorem} 
\label{couplednormalirred}
Let $\set A = \{A_{ij} \}_{i, j \in \set I}$  and $\set B = \{B_{ij} \}_{i, j \in \set I}$ where $A_{ij}$ is $n_i \times n_j$ and $B_{ij}$
is $m_i \times m_j$.   Assume \set A and~\set B are normal in the coupled sense. Suppose $S_i$ is $n_i \times m_i$ and 
$A_{ij} S_j = S_i B_{ij}$ for all $i, j$.  
\begin{enumerate}

\item If \set A and \set B are both irreducible in the coupled sense, then 
either $S_i = 0$ for all $i$, or there is a scalar $\alpha$ such that every $S_i$ is  $\alpha$ times a unitary matrix;  i.e.,
$S_i = \alpha U_i$, where $U_i$ is unitary.  In the latter case, $m_i = n_i$ for all $i$.   
Furthermore, if $\set A = \set B$, then there is a scalar~$\beta$ such that $S_i = \beta I_{n_i}$ for all $i$. 

\item
If neither \set A nor  \set B is properly reducible  in the coupled sense, then, for each $i$, either $S_i = 0$  or  $S_i$ is a scalar multiple of a unitary matrix.  
In the latter case, $m_i = n_i$.  
Furthermore, if $\set A = \set B$, then every $S_i $ is a scalar matrix.   

\item Suppose neither \set A nor \set B is strongly reducible in the coupled sense.  Assume also that 
$n_i = n$ and $m_i = m$ for all $i \in \set I$, and  that the  graph $\set G(\set A, \set B)$  is 
connected. Then either $S_i = 0$ for all $i$, or there is a scalar~$\alpha$ such that each 
$S_i$ is a~$\alpha$ times a unitary matrix;  i.e.,
$S_i = \alpha U_i$, where~$U_i$ is unitary.  In the latter case we must have $m = n$.  
Furthermore, if $\set A = \set B$, then there is some scalar $\beta$ such that $S_i = \beta I_n$ for all $i$.
\end{enumerate}  
\end{theorem}

\begin{proof}
The proofs are similar to those of Theorems~\ref{maintheorem} and~\ref{Schurlemmaversionthree}.

Suppose \set A and \set B are both irreducible in the coupled sense.  Part~1 of Theorem~\ref{maintheorem} tells us that, 
either $S_i = 0$ for all $i$, or $S_i$ is nonsingular for all $i$.  In the latter case we must have $m_i = n_i$ for all $i$.
Suppose $S_i$ is nonsingular for all $i$.    Fix $p$ and let $\lambda$ be an eigenvalue of~$S_p S_p^*$.   
Proposition~\ref{propcouplednormalfacts}  gives $A_{ij} (\set U_j(\lambda))\subseteq \set U_i(\lambda)$  for all $i, j$.  
Since \set A is irreducible in the coupled sense, either all of the subspaces $\set U_i(\lambda)$ are zero, or 
$\set U_i(\lambda) = \set V_i$ for all $i \in \set I$.  
Since $\lambda$ is an eigenvalue of $S_pS_p^*$, the space $\set U_p(\lambda)$ is nonzero.  Therefore, 
$\set U_i(\lambda) = \set V_i $ for all $i$,  and $S_i S_i^* = \lambda I_{n_i}$ for all $i$.  Since $S_i S_i^*$ is positive definite, $\lambda$ is a positive real number and 
$U_i = {1 \over {\sqrt {\lambda}}} S_i$ is a unitary matrix.  

For the second version, assume neither \set A nor \set B is properly reducible in the coupled sense.  
From part~2 of Theorem~\ref{maintheorem}, we know that, for each $i$, either $S_i = 0$ or $S_i$ is nonsingular.  If $S_i$
is nonsingular we must have $m_i = n_i$.
Suppose $S_p$ is nonsingular for some $p$.  Let $\lambda_p$ be an eigenvalue of $S_p S_p^*$.  Since~$S_p$ is nonsingular, $\lambda_p \neq 0$.
Let \set N denote the set of all $q$ such that $S_q$ is nonsingular.
Consider the statement
\begin{equation}
\label{testeqn}
 A_{ij} (\set U_j(\lambda_p)) \subseteq \set U_i(\lambda_p).
 \end{equation} 
 If $i, j$ are both in \set N, then $S_i, S_j$ are both nonsingular and Proposition~\ref{propcouplednormalfacts} tells us~(\ref{testeqn}) holds.
 If $j \notin \set N$, then $S_j = 0$, and hence, since $\lambda_p$ is nonzero,  $\set U_j(\lambda_p) = \{0\}$, so~(\ref{testeqn}) holds. 
 Finally, if $i \notin \set N$ but $j \in \set N$, then $S_i = 0$ and $S_j$ is nonsingular.    In this case, $A_{ij} S_j = S_i B_{ij}$ tells us
 $A_{ij}= 0$,  and~(\ref{testeqn}) holds. 
 Hence,~(\ref{testeqn}) holds for all $i, j$.  Since \set A is not strongly reducible in the coupled sense, there are only two possibilities for each 
 $\set U_i(\lambda)$:  it is either zero or the whole space $\set V_i$.  Since $\lambda_p$ is an eigenvalue of $S_p S_p^*$, we know $\set U_p(\lambda_p)$ is nonzero; therefore it 
 must be the whole space and $S_p S_p^* = \lambda_p I_{n_p}$.  Since $S_p S_p^*$ is positive definite, $\lambda_p$ is a positive real number and 
$U_p = {1 \over {\sqrt {\lambda_p}}} S_p$ is a unitary matrix.

Finally, consider the third version, where we assume neither \set A nor~\set B is strongly reducible in the coupled sense and $\set G(\set A, \set B)$  is 
connected. From Theorem~\ref{Schurlemmaversionfour}, either $S_i = 0$ for all $i$, or $S_i$ is nonsingular for all $i$.  In the latter case, $m = n$.

Suppose $S_i$ is nonsingular for all $i$.  
Fix $p$ and let $\lambda$ be an eigenvalue of~$S_pS_p^*$.   Since $S_p$ is nonsingular, $\lambda \neq 0$.
From Proposition~\ref{propcouplednormalfacts}, 
we have 
$A_{ij} (\set U_j(\lambda)) \subseteq \set U_i(\lambda)$, for all $i, j$, 
and the subspaces $\set U_i(\lambda)$ all have the same dimension.
 Let $f$ be the
dimension of these subspaces.  Since $\lambda$ is an eigenvalue of $S_pS_p^*$, the eigenspace $\set U_p(\lambda)$ is nonzero.  
Hence, $f > 0$.  Since \set A is not strongly reducible in the coupled sense we must have $f = n$.  Therefore
$S_i S_i^* = \lambda I_n$ for all $i$, the number $\lambda$ must be a   positive real number and 
$ U_i = {1 \over {\sqrt {\lambda}}} S_i$ is a unitary matrix.  
\end{proof}


\section{Appendix}
\label{appendix}

We construct examples to establish the claims made in Section~\ref{coupledreduc}.  

Let \set I be the index set; let  $\{n_i\}_{i \in \set I}$ be a family of positive integers.  If $n_i = 1$, set $N_i = (0)$.  
If $n_i \geq 2$, let $N_i$ be the  $n_i \times n_i$ matrix 
with a $1$ in each superdiagonal entry and zeroes elsewhere. 
This is the standard nilpotent matrix used in the blocks of the Jordan canonical form. For any $\vect x \in {\field F}^{n_i}$,  
\[
 N_i \vect x = \pmatrix{0 & 1 & 0 & \cdots & 0 \cr
0 & 0 & 1 & \cdots & 0 \cr 
\vdots & \vdots & \vdots & &  \vdots \cr
0 & 0 & 0 & \cdots & 1 \cr 
0 & 0 & 0 & \cdots & 0}   \pmatrix{x_1 \cr x_2  \cr \vdots  \cr x_{n_i-1} \cr x_{n_i}} = \pmatrix{x_2 \cr x_3 \cr \vdots \cr x_{n_i} \cr 0}.
\]
Multiplying \vect x on the left by $N_i$ moves the coordinates up one position and puts a $0$ in the last entry.  
Let $\vect e_j^i$ denote the vector with $n_i$ coordinates that has a $1$ in entry $j$ and zeroes in all other positions.  
Thus, $\vectorlist {e^i}{1}{n_i}$ are the unit coordinate vectors for $\field F^{n_i}$.  
Then $N_i \vect e_j^i = \vect e_{j-1}^i$.  Henceforth, we omit the superscript $i$ on $\vect e_j$, as the number of coordinates will 
be clear from the context.  For example, if we write $A_{ij} \vect v$, then it is understood that \vect v has $n_j$ coordinates.

Here is the key fact used
in the examples. 

\begin{proposition}
\label{Nsubspacefact}
For $n \geq 2$, let  $N$ be the $n \times n$ matrix with a $1$ in each superdiagonal entry and zeroes elsewhere. 
Suppose \set U is a nonzero, proper invariant subspace of $N$.  Then $\vect e_1 \in \set U$ and 
$\vect e_n \notin \set U$.   
\end{proposition}

\begin{proof}
Let \vect x be a nonzero vector in \set U, and let $x_k$
be the last nonzero coordinate of \vect x, i.e., $x_{k+1} = \cdots = x_n = 0$.  Then $N^{k-1} \vect x = x_k \vect e_1$, so $\vect e_1 \in \set U$.

For the second part, note that $N^{n-1} \vect e_n, N^{n-2} \vect e_n, \ldots, N \vect e_n, \vect e_n$ are the unit coordinate vectors 
\vectorlist {e}{1}{n}.  Hence, if $\vect e_n \in \set U$, then \set U is the whole space~\set V.  Since \set U is a proper subspace of \set V, 
the vector $\vect e_n$ cannot be in \set U. 
\end{proof}

\begin{remark}
Let $\set Y_j$ be the $j$-dimensional subspace spanned by  \vectorlist {e}{1}{j}, i.e.,  the 
set of all vectors with zeroes in the last $n - j$ entries.  A similar argument shows that
the nonzero invariant subspaces of~$N$ are the subspaces \scalarlist {\set Y}{n}. 
\end{remark}

We now construct some examples.  

\begin{example}
\label{propernotstrong}
Assume $\vert \set I \vert \geq 2$ and that $n_p \geq 2$ for some $p \in \set I$. 
Define \set A as follows.  
\begin{enumerate}
\item  $A_{ii} = N_i$ for all $i \in \set I$. 
\item  If $j \neq p$, set $A_{pj} = 0$. 
\item  If $i \neq p$ let $A_{ip}$ be any matrix which has $\vect e_{n_i}$ in the first column.  
\item  If $i \neq p$, and $j \neq p$, and $i \neq j$, then $A_{ij}$ can be any $n_i \times n_j$ matrix.  
\end{enumerate}
Set $\set U_i = \set V_i$ for $ i \neq p$, and let $\set U_p$ be the line spanned by $\vect e_1$.  Since $n_p \geq 2$, the 
subspace $\set U_p$ is a nonzero, proper subspace of  $\set V_p$.  
One can easily check that 
the subspaces $\{ \set U_i\}_{i \in \set I}$  properly reduce $\set A $. 
 
We now show  \set A is not strongly reducible in the coupled sense.  Suppose there were nonzero, proper subspaces $\{\set U_i\}_{i \in \set I}$ 
that reduced \set A.  (Note we must then have $n_i \geq 2$ for all $i$.)  
Each $\set U_i$ is a nonzero, proper invariant subspace of $N_i$, so 
$\vect e_1 \in \set U_i$ and $\vect e_{n_i} \notin \set U_i$.   Choose $i \neq p$. Then 
$A_{ip}$  has $\vect e_{n_i}$ in its first column,  so $A_{ip} \vect e_1 = \vect e_{n_i}$.  But $\vect e_1 \in \set U_p$ and 
$\vect e_{n_i} \notin \set U_i$, so $A_{ip}(\set U_p) \not\subseteq \set U_i$.   Hence, we have a contradiction, and \set A is not strongly reducible
in the coupled sense.  
\end{example}

Example~\ref{propernotstrong} shows \set A can be properly reducible in the coupled sense without being strongly reducible.  Thus, 
for  any field \field F,  when $\vert \set I \vert \geq 2$ and $n_i \geq 2$ for at least one $i$, we have   
$StrRed(\field F, \{n_i\}_{i \in \set I}) \subset PropRed(\field F, \{n_i\}_{i \in \set I})$.

The next example shows that if $\vert \set I \vert \geq 4$, and $n_i \geq 2$ for at least one value of $i$,  we have 
$PropRed(\field F, \{n_i\}_{i \in \set I}) \subset Red(\field F, \{n_i\}_{i \in \set I})$.
 
 \begin{example}
 \label{examplerednotproperred}
 Assume $\vert \set I \vert \geq 4$ and that $n_p \geq 2$ for some $p \in \set I$.   Choose any $q \in \set I$, with $q \neq p$, and 
 define \set A as follows.  
 \begin{enumerate}
 \item $A_{ii} = N_i$ for all $i \in \set I$. 
 \item For all $i$ with $i \neq p$ and $i \neq q$, set  $A_{ip} = 0$ and $A_{iq} = 0$.
 \item For all other choices of $i, j$ with $i \neq j$, let $A_{ij}$ be any matrix with $\vect e_{n_1}$ in the first column.
  \end{enumerate}
 
 We illustrate for $\set I = \{1, 2, \ldots, K\}$, with $p = 1$ and $q = 2$.  
 \[
 A = \pmatrix{ N_1 & * & \Big| & * & * & \cdots & * \cr 
 * & N_2 & \Big| & * &* & \cdots & * \cr 
 \hline \cr
0 & 0 & \Big| & N_3 & *  & \cdots & * \cr
0 &0 & \Big| &* & N_4   & \cdots & * \cr
\vdots & \vdots & \Bigg|  & \vdots  & \vdots  & \ddots &  \vdots \cr 
0 & 0 & \Big|  & * & * & \cdots & N_K },
\]
where each asterisk ($*$)  represents an $n_i \times n_j$ matrix with $\vect e_{n_i}$ in the first column.

Set $\set U_p = \set V_p$, and $\set U_q = \set V_q$.  For all other values of $i$,   set $\set U_i = 0$.
One can check that \set A is coupled reducible via $\{ \set U_i\}_{i \in \set I}$.  

We now show \set A is not properly reducible.  Suppose \set A could be properly reduced by subspaces $\{ \set U_i\}_{i \in \set I}$.
At least one  $\set U_i$ must be a nonzero, proper subspace; we first show this holds for at most one value of $i$.  
Suppose $\set U_i$ and~$\set U_j$  were both nonzero, proper subspaces, with $i \neq j$.  We must then have $n_i \geq 2$ and $n_j \geq 2$.  
Since 
$\set U_i$ is a nonzero, proper invariant subspace of~$N_i$, and~$\set U_j$ is a nonzero, proper invariant subspace of~$N_j $, we know $\vect e_1$ is in 
both~$\set U_i$ and $\set U_j$, and $\vect e_{n_i} \notin \set U_i$ and $\vect e_{n_j} \notin \set U_j$.
If $i = p$ and $j = q$, use the matrix~$A_{pq}$.  Since $A_{pq} \vect e_1 = \vect e_{n_p}$, we see 
$A_{pq}(\set U_q) \not\subseteq \set U_p$.  The same argument, using $A_{qp}$,  applies when $i = q$ and $j = p$. 
Suppose then that at least one of $i, j$ is different from $p$ and $q$.  Without loss of generality, assume $j \not\in \{p, q\}$.
Then use $A_{ij}$, which has $\vect e_{n_i}$ in its first column.  So $A_{ij} (\set U_j) \not\subseteq \set U_i$.  

So, at most one $\set U_i$ is a proper, nonzero subspace; each of the
other subspaces is either the whole space $\set V_i$ or the zero subspace.  Assume $\set U_i$ is the nonzero, proper subspace; note $n_i \geq 2$.  
We claim we can 
then choose $j \neq i$  so that $A_{ij}$ has~$\vect e_{n_i}$ in the first column and $A_{ji}$ has $\vect e_{n_j}$ in the first column.  
If $i = p$, choose $j = q$, and if $i = q$, choose $j = p$.  
If $i \neq p$ and $i \neq q$, choose any $j$ which is different from $i, p$ and $q$.  (This is where we use the fact that $\vert \set I \vert  \geq 4$.)   
The subspace $\set U_j$ is either the full space $\set V_j$, or it is the zero subspace. 
If $\set U_j = \set V_j$, then $A_{ij}( \set V_j)$ contains $A_{ij} \vect e_1 = \vect e_{n_i}$, which is not  in $\set U_i $.  
So $A_{ij} (\set U_j) \not\subseteq \set U_i$.
If $\set U_j = \{0\}$, 
then, since  $\vect e_1 \in \set U_i$, we have $A_{ji} \vect e_1 = \vect e_{n_j}  \in A_{ij} (\set U_i)$.  So $A_{ji} (\set U_i) \not\subseteq \set U_j$.
Hence,  \set A is not properly reducible in the coupled sense.  
 \end{example}

 In the example above, we needed $\vert \set I \vert  \geq 4$.  What can we say when $K = 2$ or $K = 3$?  
 In these cases, the field \field F must be considered.  The reason is, that for \set A to be properly reducible in the coupled sense, at least one 
 $A_{ii}$ must have a nonzero, proper invariant subspace.  If \field F is algebraically closed and $n \geq 2$, then any $n \times n$ matrix over \field F has an eigenvalue in 
 \field F, and the line spanned by a corresponding eigenvector is a nonzero, proper invariant subspace.  But if \field F is not algebraically closed, there may be 
 $n \times n$ matrices over \field F which have no proper invariant subspaces.  We shall give an example for the real numbers later, but   
 first we show that if \field F is an algebraically closed field, then 
 $PropRed(\field F, n, 2) = Red(\field F, n, 2)$ and $PropRed(\field F, n, 3) = Red(\field F, n, 3)$ for all $n \geq 2$. 
 
We use the following lemma  to deal with the cases $K = 2$ and $K = 3$.  
 
 \begin{lemma} 
 \label{subspacelemmaextreme}
 Let $\set A = \{A_{ij} \}_{i, j \in \set I}$ where $A_{ij}$ is $n_i \times n_j$. 
 Suppose \set A is coupled reducible  with 
$\{\set U_i\}_{i \in \set I}$ satisfying one of the following.
 \begin{enumerate}
 \item $\set U_p = \set V_p$ for exactly one index value $p$, and  $\set U_i = \{0\}$ when $i \neq p$.
 \item $\set U_p = \{0 \}$ for exactly one index value $p$ and $\set U_i =\set V_i$ when $i \neq p$.
 \end{enumerate}
 Suppose $\set W_p$  is a nonzero, proper invariant subspace of $A_{pp}$. 
 Then  \set A is properly reducible by coupled similarity via the subspaces obtained by replacing $\set U_p$ by $\set W_p$, and 
 leaving the other $\set U_i$'s unchanged.  

 \end{lemma}
 
 \begin{proof}
Since $\set U_p$ is the only subspace that is changed, we continue to have $A_{ij}(\set U_j) \subseteq \set U_i$ whenever $i$ and $j$ are both different from $p$.   
Also, $\set W_p$  is chosen to satisfy  $A_{pp}(\set W_p) \subseteq \set W_p$.
It remains to consider $A_{ip}$ and $A_{pi}$  for $i \neq p$.  
 
 In case 1, we have  $\set U_i = \{0 \}$ for $i \neq p$, so $A_{pi} (\set U_i) =  \{0 \} \subseteq \set W_p$.  We also have 
 $A_{ip} (\set U_p)  \subseteq \set U_i = \{0 \}$.   Since  $\set U_p = \set V_p$, we must have 
 $A_{ip} (\set W_p) = \{0 \} = \set U_i $.
 
 In case 2, we have $\set U_i = \set V_i$ for $i \neq p$, and $\set U_p =  \{0 \}$.   So $A_{pi}(\set U_i) = \{0\} \subseteq \set W_p$.     
 We also have  $A_{ip}(\set W_p) \subseteq \set V_i = \set U_i$ for $i \neq  p$. 
 \end{proof}

 Now suppose  \field F is algebraically closed, and $n \geq 2$.   Any $n \times n$ matrix over \field F has a nonzero, proper invariant subspace.  
 For $K = 2$, Lemma~\ref{subspacelemmaextreme} immediately tells us that \set A is 
coupled reducible if and only if it is properly reducible, i.e., $PropRed(\field F, n, 2) = Red(\field F, n, 2)$. 
For the case $K= 3$, suppose \set A is reduced by $\set U_1, \set U_2, \set U_3$. If none of the $\set U_i$'s is a nonzero proper subspace, then each is either 
 \set V or $0$, so either two of them are \set V, with the third being zero, or vice versa, two of them are  zero, with the third being~\set V.
Lemma~\ref{subspacelemmaextreme} then tells us \set A is properly reducible.  Hence, for algebraically closed
 \field F and $n \geq 2$ we have 
 $PropRed(\field F, n, 3) = Red(\field F, n, 3).$

If \field F is not algebraically closed, then a matrix over \field F need not have a proper invariant subspace.  Consider the case $\field F =  \blabold R$, the field of real numbers. 
Let $A$ be a real $n \times n$ matrix, where $n \geq 2$.  
The eigenvalues of $A$ are in \blabold C, and the non-real eigenvalues occur in conjugate pairs.  If $\lambda$ is a real eigenvalue of 
$A$ then there is a corresponding real eigenvector, \vect v, and the line spanned by \vect v is a proper, nonzero invariant subspace of $A$.  
For a pair of complex conjugate, non-real eigenvalues, $\lambda, \overline\lambda$, there is  a corresponding two dimensional invariant subspace.

Consider the following example for $K \geq 2$ and 
$2 \times 2$ real matrices.

\begin{example}
\label{examplereal}
Choose an angle $\theta$ with $0 < \theta < \pi$.  For $1 \leq i \leq K$, set 
\[ 
A_{ii}  = \pmatrix{ \cos \theta & - \sin \theta \cr \sin \theta & \cos \theta}.
\]
 This is the matrix for rotation of the plane ${\blabold R}^2$ by angle $\theta$.  
Since no line through the origin is mapped to itself by this rotation, this map has no nonzero, proper invariant subspace.  Hence, for any choice of the $A_{ij}$'s when 
$i \neq j$,  the set \set A is not properly reducible in the coupled sense.  It is, however, possible to find $A_{ij}$'s such that \set A is reducible in the coupled sense.  
Choose a positive integer $s$ with $1 \leq s < K$ and set $A_{ij} = 0$ whenever $i > s$ and $j \leq s$.  Set $\set U_i = \blabold R^2$ for 
$1 \leq i \leq s$ and $\set U_i = \{0\}$ for $s+1 \leq i \leq K$.  It is easy to check that the subspaces \scalarlist {\set U}{K} reduce \set A.  For, when 
$i $ and $j$ are both less than or equal to $s$, we have $\set U_i = \set U_j = \blabold R^2$, and hence $A_{ij} (\set U_j) \subseteq \set U_i$.  If $i$ and $j$ are 
both greater than~$s$, then $\set U_i = \set U_j = \{0\}$, so $A_{ij} (\set U_j) \subseteq \set U_i$. If $i > r$ and $j \leq s$, then $A_{ij} = 0$; hence
$A_{ij} (\set U_j) = \{0 \}  \subseteq \set U_i$.  Finally, if $i \leq s$ and $j > s$, then $\set U_j = \{0\}$ so $A_{ij} (\set U_j) = \{0\} \subseteq \set U_i$.  
So \set A is reducible in the coupled sense, but not properly reducible. 
\end{example} 

So for $K \geq 2$, we have $PropRed(\blabold R, 2, K) \subset Red(\blabold R, 2, K)$.  
From Example~\ref{examplerednotproperred}, 
we already knew this for $K \geq 4$; the new information is that 
$PropRed(\blabold R, 2, 2) \subset Red(\blabold R, 2, 2)$ and 
$PropRed(\blabold R, 2, 3)  \subset Red(\blabold R, 2, 3)$.

However, for $n \geq 3$, any $n \times n$ real matrix has a nonzero proper invariant subspace. 
Lemma~\ref{subspacelemmaextreme} then gives 
$PropRed(\blabold R, n, 2) = Red(\blabold R, n, 2)$ and $PropRed(\blabold R, n, 3) = Red(\blabold R, n, 3)$
when $n \geq 3$.  

\

\noindent
 {\bf Acknowledgements.}  D. Lahat thanks Dr. Jean-Fran\c{c}ois Cardoso for
insightful discussions during her Ph.D., which set the foundations for and motivated this work.

\end{document}